\documentclass[letterpaper]{amsart}

\usepackage[utf8]{inputenc}
\usepackage{lmodern}
\usepackage[T1]{fontenc}
\usepackage{amssymb}
\usepackage{mathtools}
\usepackage{graphicx}
\usepackage[pdfusetitle]{hyperref}
\usepackage{breakurl}

\def\doi#1{\burlalt{http://doi.org/#1}{doi:#1}}
\usepackage{tikz-cd}

\mathtoolsset{mathic=true}

\theoremstyle{plain}
\newtheorem{theorem}{Theorem}[section]
\newtheorem{proposition}[theorem]{Proposition}
\newtheorem{lemma}[theorem]{Lemma}
\newtheorem{corollary}[theorem]{Corollary}

\theoremstyle{definition}

\newtheorem{remark}[theorem]{Remark}
\newtheorem{question}[theorem]{Question}
\newtheorem{assumption}[theorem]{Assumption}
\newtheorem{conjecture}[theorem]{Conjecture}

\theoremstyle{remark}
\newtheorem*{acknowledgements}{Acknowledgements}

\numberwithin{equation}{section}

\def\Z{\mathbb Z}
\def\Q{\mathbb Q}
\def\R{\mathbb R}
\def\C{\mathbb C}

\DeclareMathOperator{\Hom}{Hom}
\DeclareMathOperator{\Tor}{Tor}
\DeclareMathOperator{\coker}{coker}
\DeclareMathOperator{\im}{im}

\def\setminus{\mathbin{\mkern-2mu\MFsm\mkern-2mu}}

\let\kk\Bbbk
\DeclareMathOperator{\depth}{depth}
\def\pair#1{{\langle#1\rangle}}
\def\Hc{H_{c}}
\def\HT{H_{T}}

\def\CTc{C_{T,c}}

\def\hHc{H^{c}}
\def\hCTc{C^{T,c}}
\def\hHT{H^{T\!}}
\def\hHTc{H^{T,c}}

\def\ie{\emph{i.\,e.}}
\def\cf{\emph{cf.}}

\let\setminus\smallsetminus
\let\obigwedge\bigwedge
\def\bigwedge{\obigwedge\nolimits}
\def\RR{R}
\def\UU{\bar V}
\def\dV{d}
\def\dU{\bar d}
\def\Vr{V^{r}}
\def\Ur{\UU ^{r}}

\def\pp{\mathfrak{p}}
\def\PP{\mathcal P}
\def\KK#1{K_{\bb,#1}}
\def\contr{\mathbin{\raisebox{1pt}{\rotatebox[origin=c]{180}{$\neg$}}}}
\let\epsilon\varepsilon
\def\BC{\mathcal{BC}}
\DeclareMathOperator{\syzord}{syzord}

\def\aa{a}
\def\bb{b}
\def\ttt{t}
\def\PD{PD}
\def\shufflesign#1#2{(-1)^{(#1,#2)}}

\begin{document}

\title{Big polygon spaces}
\author{Matthias Franz}
\address{Department of Mathematics, University of Western Ontario,
      London, Ont.\ N6A\;5B7, Canada}
\email{mfranz@uwo.ca}

\thanks{The author was partially supported by an NSERC Discovery Grant.}


\begin{abstract}
  We study a new class of compact orientable manifolds, called big polygon spaces.
  They are intersections of real quadrics and related to polygon spaces, which appear as their fixed point set
  under a canonical torus action.

  What makes big polygon spaces interesting is that they exhibit remarkable new features
  in equivariant cohomology:
  The Chang--Skjelbred sequence can be exact for them and
  the equivariant Poincaré pairing perfect although their equivariant cohomology is never free
  as a module over the cohomology ring of~\(BT\).
  More generally, big polygon spaces show
  that a certain bound on the syzygy order of the equivariant cohomology of compact orientable \(T\)-manifolds
  obtained by Allday, Puppe and the author is sharp.
\end{abstract}

\maketitle
\section{Introduction}

{\let\kk\Q
Let \(\RR =\kk[\ttt_{1},\dots,\ttt_{r}]\) be a polynomial ring.
A finitely generated \(\RR \)-module~\(M\) is called an \emph{\(m\)-th syzygy}
if there is an exact sequence
\begin{equation}
  \label{eq:def-syzygy}
  0\to M\to F_{1}\to\dots\to F_{m}
\end{equation}
with finitely generated free modules~\(F_{1}\),~\ldots,~\(F_{m}\).
The first syzygies are exactly the torsion-free modules and the \(r\)-th syzygies the free ones.
In this sense, syzygies interpolate between torsion-freeness and freeness.
We also call any~\(M\) a zeroeth syzygy.
If \(M\) is an \(m\)-th syzygy, but not one of order~\(m+1\) (or if \(m=r\)),
then we say that it is of order exactly~\(m\), and we write \(\syzord M=m\).

Allday, Puppe and the author
have initiated the study of syzygies in the context of torus-equivariant cohomology
\cite{AlldayFranzPuppe:orbits1},~\cite{AlldayFranzPuppe:orbits4}.
Let \(T=(S^{1})^{r}\) be a torus, and
let \(X\) be a \(T\)-manifold
such that its
rational cohomology~\(H^{*}(X)\)
is finite-dimensional. Then \(\RR =H^{*}(BT)\), the cohomology of the classifying space of~\(T\),
is of the form described above with generators of degree~\(2\),
and the (Borel) equivariant cohomology~\(\HT^{*}(X)\) is a module, even an algebra, over~\(\RR \).
Many authors have investigated the cases where \(\HT^{*}(X)\) is torsion-free or free.
One insight of~\cite{AlldayFranzPuppe:orbits1} was that reflexive \(\RR \)-modules
(the second syzygies) are equally important: \(\HT^{*}(X)\) is reflexive
if and only if the Chang--Skjelbred sequence
\begin{equation}
  \label{eq:6:chang-skjelbred}
  \let\longrightarrow\rightarrow
  0
  \longrightarrow \HT^{*}(X)
  \longrightarrow \HT^{*}(X^{T})
  \longrightarrow \HT^{*+1}(X_1, X^{T})
\end{equation}
is exact, where \(X_{1}\) denotes the union of the fixed point set~\(X^{T}\) and the \(1\)-dimensional orbits
\cite[Thm.~1.1]{AlldayFranzPuppe:orbits1}.
This often permits an efficient computation of~\(\HT^{*}(X)\).
Moreover, if \(X\) is compact oriented with equivariant orientation~\([X]_{T}\),
then the equiva\-riant Poincaré pairing
\begin{equation}
  \label{eq:6:poincare-cup-equiv-intro}
  \HT^{*}(X)\times \HT^{*}(X) \to \RR ,
  \quad
  (\alpha,\beta) \mapsto
  \pair{\alpha\cup\beta,[X]_{T}},
\end{equation}
is perfect if and only if \(\HT^{*}(X)\) is reflexive \cite[Cor.~1.3]{AlldayFranzPuppe:orbits1}.

Syzygies of any order can appear as the equivariant cohomology of \(T\)-manifolds.
Assume for example that \(X\) is a (compact) toric manifold; it is well-known
that \(\HT^{*}(X)\) is free over~\(\RR \) in this case. By removing two fixed points
from~\(X\), one can obtain, for any~\(m<r\), syzygies of order exactly~\(m\) \cite[Sec.~6.1]{AlldayFranzPuppe:orbits1},~\cite[Sec.~6.1]{Franz:orbits3}.
The situation changes dramatically in the presence of Poincaré duality \cite[Cor.~1.4]{AlldayFranzPuppe:orbits1}:

\begin{theorem}[Allday--Franz--Puppe]
  \label{thm:syzygy-PD-space}
  Let \(X\) be a 
  compact orientable \(T\)-manifold.
  If \(\HT^{*}(X)\)
  is a syzygy of order \(m\ge r/2\), then it is free over~\(\RR \).
\end{theorem}

The aim of this note is to show that this bound is sharp.
For~\(r\in\{3,5,9\}\), Puppe and the author~\cite{FranzPuppe:2008}
have previously constructed compact orientable \(T\)-manifolds~\(X\)
with \(\syzord\HT^{*}(X)=1\). 
A modest generalization of the construction appeared in~\cite[Sec.~6.2]{Franz:orbits3}.
However, no examples were known so far of rational Poincaré duality spaces~\(X\)
such that \(\HT^{*}(X)\) is reflexive, 
but not free.
By providing these, we now in particular give the first examples of rational Poincaré duality spaces~\(X\)
such that \(\HT^{*}(X)\) is not free over~\(\RR \),
but such that the Chang--Skjelbred sequence~\eqref{eq:6:chang-skjelbred} is exact
and the equivariant Poincaré pairing~\eqref{eq:6:poincare-cup-equiv-intro} perfect.

\medskip

A vector~\(\ell\in\R^{r}\) is called \emph{generic}
if one cannot split up its components into two groups of equal sum.
For generic~\(\ell\) and \(\aa\),~\(\bb\ge1\)
consider the real algebraic variety~\(X_{\aa,\bb}(\ell)\subset\C^{r(\aa+\bb)}\) defined by the equations
\begin{alignat}{2}
  \label{eq:def-X-quadratic}
  \|u_{j}\|^{2} + \|z_{j}\|^{2} &= 1 &\qquad& (1\le j\le r), \\
  \label{eq:def-X-linear}
  \ell_{1}u_{1}+\dots+\ell_{r}u_{r} &= 0,
\end{alignat}
where \(u_{1}\),~\ldots,~\(u_{r}\in\C^{\aa}\) and \(z_{1}\),~\ldots,~\(z_{r}\in\C^{\bb}\).
We call \(X_{\aa,\bb}(\ell)\) a \emph{big polygon space}.
The torus~\(T=(S^{1})^{r}\) acts on it by scalar multiplication
on the variables~\(z_{j}\),
\begin{equation}
  (g_{1},\dots,g_{r})\cdot(u_{1},\dots,u_{r},z_{1},\dots,z_{r}) = (u_{1},\dots,u_{r},g_{1}z_{1},\dots,g_{r}z_{r}).
\end{equation}
Given that \(\ell\) is generic, \(X_{\aa,\bb}(\ell)\) is an orientable compact connected \(T\)-manifold.
Permuting the coordinates of~\(\ell\) or changing their signs
does not produce new equivariant diffeomorphism types, so one can always
assume the components of~\(\ell\) to be non-negative and weakly increasing (Lemma~\ref{thm:X-ell-mf}).

If \(\ell\) has positive components, one can think of~\(X_{\aa,\bb}(\ell)\)
as the set of all \(r\)-tuples of vectors in~\(\C^{\aa+\bb}\) of lengths~\(\ell_{1}\),~\ldots,~\(\ell_{r}\)
whose sum lies on a fixed \(\bb\)-dimensional complex subspace.
The set of \(T\)-fixed points
corresponds to setting \(z=0\),
which gives all \(r\)-tuples of vectors in~\(\C^{\aa}\)
of lengths~\(\ell_{1}\),~\ldots,~\(\ell_{r}\)
which form a polygon in the sense that they add up to~\(0\).
This is an example of a ``space of polygons''.
Various kinds of polygon spaces have been studied
by Walker, Hausmann, Klyachko, Kapovich, Millson, Knutson, Farber, Schütz, Fromm and others,
see \cite{FarberFromm:2013}, \cite{FarberSchuetz:2007}, \cite{Hausmann:1991},~\cite[\S 10.3]{Hausmann:mod2}
and the references given therein.

Our main result says that for any choice of~\(\aa\),~\(\bb\) and~\(r\)
there is a big polygon space~\(X_{\aa,\bb}(\ell)\)
that produces a maximal non-free syzygy in equivariant cohomology,
and it is essentially unique.

\begin{theorem}
  \label{thm:main}
  Let \(\aa\),~\(\bb\),~\(r\ge1\), and
  let \(\ell\in\R^{r}\) be generic
  with~\(0\le\ell_{1}\le\dots\le\ell_{r}\).
  \begin{enumerate}
  \item Assume \(r=2m+1\). Then
    \(\syzord\HT^{*}(X_{\aa,\bb}(\ell))=m\)
    if and only if
    \(X_{\aa,\bb}(\ell)\) is equivariantly diffeomorphic to~\(X_{\aa,\bb}(1,\dots,1)\).
  \item Assume \(r=2m+2\). Then
    \(\syzord\HT^{*}(X_{\aa,\bb}(\ell))=m\)
    if and only if
    \(X_{\aa,\bb}(\ell)\) is equivariantly diffeomorphic to~\(X_{\aa,\bb}(0,1,\dots,1)\).
  \end{enumerate}
\end{theorem}

This implies that
all syzygy orders less than~\(r/2\) can be realized via big polygon spaces
(Corollary~\ref{thm:HTX-equilateral-0}).

\smallskip

The proof of Theorem~\ref{thm:main} appears in Sections~\ref{sec:equilateral} and~\ref{sec:syzygy}.
Before, we discuss generalities of big polygon spaces (Section~\ref{sec:big-chain}) and their cohomology,
first non-equivariant (Section~\ref{sec:betti}, including an analogue of Walker's conjecture) and then equivariant (Section~\ref{sec:proof-main}).
We conclude with several additional comments in Section~\ref{sec:comments}.%
}

Unless specified otherwise, all (co)homology in this paper
is taken with coefficients in a field~\(\kk\) of characteristic~\(0\),
and all tensor products are over~\(\kk\).
All manifolds we consider are assumed to be smooth and
to have finite-dimensional cohomology.
Products of oriented manifolds are oriented in the canonical way
according to the order of the factors.
We orient the unit sphere~\(S\subset\C^{n}\) such that
the induced orientation on~\((0,\infty)\times S\) agrees with the canonical orientation on~\(\C^{n}\).

\begin{acknowledgements}
  It is a pleasure to thank Volker Puppe for many stimulating discussions.
  His calculations with the minimal Hirsch--Brown model
  indicated the relevance of \(T\)-spaces~\(X\) such that the middle Betti numbers
  of~\(X\) and~\(X^{T}\) are binomial coefficients
  as given by Proposition~\ref{thm:betti-Xell} and formula~\eqref{eq:poincare-polygon}.
  He also suggested several improvements of the exposition
  as well as the name `big polygon space' for~\(X_{\aa,\bb}(\ell)\),
  in analogy with the big chain spaces introduced in~\cite{Hausmann:mod2}.
  Moreover, I am indebted to Vinicio Gómez Gutiérrez
  and Santiago López de Medrano for informing me
  about their observation that the ``mutants'' constructed in~\cite{FranzPuppe:2008}
  might be realizable as intersections of quadrics
  of the form~\eqref{eq:two-real-quadrics} and~\eqref{eq:unit-sphere}.
  I finally thank Michael Farber for stimulating discussions,
  Sean Fitzpatrick, John Malik
  and the referees for their careful reading of previous versions of this paper,
  and particularly Sergio Chaves for pointing out previously incorrect degrees shifts in Proposition~\ref{thm:HTX-equilateral}.
\end{acknowledgements}

\section{First properties}
\label{sec:big-chain}

Let \(r\ge1\). We use the abbreviation \([r]=\{1,\dots,r\}\),
and for a subset~\(J\subset[r]\) we write
\(|J|\) for the size of~\(J\), \(J^{c}=[r]\setminus J\)
and \(J\cup j\) instead of~\(J\cup\{j\}\) etc.\ for \(j\in[r]\).
Moreover, given two disjoint subsets~\(J\),~\(K\subset[r]\),
we denote the sign of the shuffle~\((J,K)\) by
\begin{equation}
  \shufflesign{J}{K} = (-1)^{|\,\{(j,k)\in J\times K \,\mid\, j > k\}\,|} \, .
\end{equation}

For a vector~\(\ell\in\R^{r}\), called \emph{length vector} in this context,
and \(J\subset[r]\) we define
\begin{equation}
  \ell(J) = \sum_{j\in J}\ell_{j}.
\end{equation}
One says that \(\ell\) is \emph{generic} if
\begin{equation}
  \label{eq:def-generic}
  \forall\, J\subset[r] \quad \ell(J) \ne \ell(J^{c}).
\end{equation}
In this case \(J\) is called \emph{\(\ell\)-long} or \emph{\(\ell\)-short},
depending on whether the left or the right hand side dominates in~\eqref{eq:def-generic}.
If \(\ell\) is clear from the context, we just say `long' or `short'.
The non-generic length vectors lie on hyperplanes given by normal vectors
with coordinates equal to~\(\pm1\).
The connected components of the complement of
this hyperplane arrangement are called \emph{chambers}. 
Two length vectors~\(\ell\) and~\(\ell'\) lie in the same chamber
if and only if they induce the same notion of `long' and `short';
we write \(\ell\sim\ell'\) in this case.
Permuting components of~\(\ell\) or changing their signs does not affect genericity.

The big polygon space~\(X_{\aa,\bb}(\ell)\) for \(\aa\),~\(\bb\),~\(r\ge1\) and generic~\(\ell\in\R^{r}\)
as well as the action of \(T=(S^{1})^{r}\) on it have been defined in the introduction.
Note that unlike polygon spaces, big polygon spaces are non-empty for any~\(\ell\).
We write points of~\(X_{\aa,\bb}(\ell)\) or the ambient space~\(\C^{r(\aa+\bb)}\) in the form~\((u,z)\).
The fixed point set of~\(X_{\aa,\bb}(\ell)\) is the space of polygons~\(E_{2\aa}(\ell)\) studied
in~\cite{FarberSchuetz:2007},~\cite{Fromm:2011},~\cite{FarberFromm:2013} and~\cite[\S 10.3]{Hausmann:mod2}.
Note that if \(\ell\) has a zero coordinate, say \(\ell=(0,\ell')\), then
\(\ell'\) is again generic, and there is an equivariant diffeomorphism
\begin{equation}
  \label{eq:ell-0}
  X_{\aa,\bb}(\ell)\cong S^{2\aa+2\bb-1}\times X_{\aa,\bb}(\ell')
\end{equation}
where the \(S^{1}\)-action on~\(S^{2\aa+2\bb-1}\subset\C^{\aa+\bb}\) comes from scalar multiplication in~\(\C^{\bb}\).

\begin{lemma}
  Let \(\ell\) and \(\ell'\) be generic length vectors in~\(\R^{r}\).
  \label{thm:X-ell-mf}
  \begin{enumerate}
  \item
    \label{thm:X-ell-mf-1}
    \(X_{\aa,\bb}(\ell)\) is an orientable compact connected \(T\)-manifold; its dimension is\linebreak \((2\aa+2\bb-1)r-2\aa\).
  \item
    \label{thm:X-ell-mf-2}
    If \(\ell'\) is obtained from~\(\ell\) by changing the sign of some components
    and{\slash}or by permuting them, then \(X_{\aa,\bb}(\ell)\) and \(X_{\aa,\bb}(\ell')\) are diffeomorphic,
    equivariantly with respect to the corresponding permutation of the components of~\(T\).
  \item
    \label{thm:X-ell-mf-3}  
    If \(\ell\sim\ell'\),
    then \(X_{\aa,\bb}(\ell)\) and~\(X_{\aa,\bb}(\ell')\) are equivariantly diffeomorphic.
  \end{enumerate}
\end{lemma}

\begin{proof}
  It is clear that \(X_{\aa,\bb}(\ell)\) is compact. Also, by scaling \(u\) by~\(\lambda\in[0,1]\)
  and suitably increasing the variables~\(z_{j}\), one can connect any point~\((u,z)\in X_{\aa,\bb}(\ell)\)
  to the subset~\(X_{\aa,\bb}(\ell)\cap\{u=0\}\cong T\).
  Hence \(X_{\aa,\bb}(\ell)\) is path-connected.
  
  To complete the proof of~\eqref{thm:X-ell-mf-1},
  we show that \(X_{\aa,\bb}(\ell)\) is an orientable \(T\)-submanifold of~\(\C^{r(\aa+\bb)}\);
  for this it suffices to verify that \(0\) is a regular value
  of the \(T\)-invariant function
  \begin{equation}
  \begin{split}
    F_{\ell}\colon\C^{r(\aa+\bb)} &\to \R^{r}\times\C^{\aa}, \\
    (u,z) &\mapsto
    \bigl(\|u_{1}\|^{2} + \|z_{1}\|^{2}-1,\dots,\|u_{r}\|^{2} + \|z_{r}\|^{2}-1, \sum_{j=1}^{r}\ell_{j}u_{j} \bigr).
  \end{split}
  \end{equation}
  We start with the case where all components of~\(\ell\) are non-negative.
  Because the function~\(\C^{\aa+\bb}\supset S^{2\aa+2\bb-1}\to\R\), \((u_{j},z_{j})\mapsto \|u_{j}\|^{2} + \|z_{j}\|^{2}-1\)
  is submersive, we may by induction assume that all components of~\(\ell\) are actually positive.
  In this case
  it is easy to see that the differential~\(D F_{\ell}(u,z)\) is surjective
  for all~\((u,z)\in X_{\aa,\bb}(\ell)\) such that~\(z\ne0\).
  If \(z=0\), then we are inside the space of polygons~\(E_{2\aa}(\ell)\),
  and the argument in~\cite[Thm.~3.1]{Hausmann:1991} (or~\cite[Prop.~3.1]{FarberFromm:2013}) applies.

  If \(\ell'\) is obtained from~\(\ell\) as in~\eqref{thm:X-ell-mf-2},
  then changing the sign of some~\(u_{j}\) and{\slash}or permuting the pairs~\((u_{j},z_{j})\)
  defines an automorphism of~\(\C^{r(\aa+\bb)}\) that carries \(X_{\aa,\bb}(\ell)\) to~\(X_{\aa,\bb}(\ell')\);
  this automorphism is equivariant with respect to the corresponding permutation of the components of~\(T\).
  This proves \eqref{thm:X-ell-mf-1} for general~\(\ell\) and also part~\eqref{thm:X-ell-mf-2}.

  The last claim follows from an equivariant version of the Ehresmann fibration theorem:
  Let \(C\) be a chamber and consider the map
  \begin{equation}
    F\colon\C^{r(\aa+\bb)}\times C \to \R^{r}\times\C^{\aa}\times C, \qquad
    (u,z,\ell) \mapsto (F_{\ell}(u,z), \ell).
  \end{equation}
  The first part implies that \(F\) is again a submersion, hence
  so is the restriction \(\tilde F\colon F^{-1}(0,C)\to C\),
  whose fibre over~\(\ell\in C\) is \(X_{\aa,\bb}(\ell)\).
  Now take a non-vanishing vector field~\(\xi\) on the line segment~\(L\subset C\) connecting \(\ell\) and~\(\ell'\)
  and lift it to a vector field~\(\tilde\xi\) on~\(\tilde F^{-1}(L)\) which is perpendicular to~\(\ker T\tilde F\)
  with respect to some \(T\)-invariant metric.
  From the flow of~\(\tilde\xi\) we get a \(T\)-equivariant diffeomorphism between~\(X_{\aa,\bb}(\ell)\)
  and \(X_{\aa,\bb}(\ell')\).
\end{proof}

\begin{assumption}
  We assume for the rest of this paper that all length vectors~\(\ell\) we consider
  are generic and of the form~\(0\le\ell_{1}\le\dots\le\ell_{r}\).
\end{assumption}

This is justified by part~\eqref{thm:X-ell-mf-2} of Lemma~\ref{thm:X-ell-mf},
given that we are only concerned with cohomological features of the big polygon spaces. 
Because the chambers are open in~\(\R^{r}\),
we may by~\eqref{thm:X-ell-mf-3} even assume the~\(\ell_{j}\) to be positive and strictly increasing.
This will sometimes be convenient.

\section{Non-equivariant cohomology}
\label{sec:betti}

We now compute the non-equivariant cohomology of~\(X_{\aa,\bb}(\ell)\);
this will also serve as a warm-up for the equivariant situation in the next section.
Our general approach 
is modelled on that of Farber--Fromm~\cite{Fromm:2011},~\cite{FarberFromm:2013}.
The (equivariant) perfection of the Morse--Bott function however will follow
from a simple symmetry argument.

Let \(\aa\),~\(\bb\),~\(r\ge1\).
We write \(X=X_{\aa,\bb}(\ell)\)
and introduce the abbreviations
\begin{alignat}{2}
  V &= S^{2\aa+2\bb-1}\subset\C^{\aa}\times\C^{\bb}, &\qquad \dV &= \dim V=2\aa+2\bb-1,\\
  \UU  &= V\cap (\,\C^{\aa}\times0\,) = S^{2\aa-1}, &\qquad \dU &= \dim \UU =2\aa-1.
\end{alignat}
We choose a base point~\(*\in \UU \subset V\)
and define for \(J\subset[r]\) the Cartesian product
\begin{align}
  \label{eq:def-VJ}
  V_{J} &= \{\, (u,z)\in V^{r} \mid \forall j\notin J \;\;\; (u_{j},z_{j})=* \,\} \\
\shortintertext{and for short~\(J\) also}
  \label{eq:def-WJ}
  W_{J} &= \{\, (u,z)\in V^{r} \mid \forall i,j\notin J \;\;\; u_{i}=u_{j}, \; z_{i} = z_{j} = 0 \,\} \cong V_{J}\times \UU, \\
  P_{J} &= \bigl\{ u\in \UU ^{r} \mid
  \forall j\in J,\, i,k\notin J \;\;\; u_{j}=-u_{i}=-u_{k} \,\bigr\} \cong\UU .
\end{align}
Then \(V_{J}\subset W_{J}\subset \Vr \setminus X\) and \(P_{J}\subset (W_{J})^{T}\), moreover
\begin{alignat}{2}
  \label{eq:dim-VJ-WJ}
  \dim V_{J} &= |J|\dV, &\qquad
  \dim W_{J} &= |J|\dV+\dU .
\end{alignat}
(Our \(V_{J}\) and \(W_{J}\) correspond to \(V_{J^{c}}\) and \(W_{J^{c}}\)
in the notation of~\cite[p.~3105]{FarberFromm:2013}.)

We orient the manifolds \(V_{J}\) and \(W_{J}\) as in~\cite[p.~71]{Fromm:2011}:
Let \(J=\{j_{1}<\dots<j_{k}\}\).
The orientations of~\(V\) and~\(\UU\)
give canonical orientations of~\(V^{k}\) and \(V^{k}\times \UU \)
according to the order of the factors. We transport these orientations
to \(V_{J}\) and \(W_{J}\) via the diffeomorphisms
\begin{alignat}{2}
  V_{J} &\to V^{k}, &\qquad  v &\mapsto (v_{j_{1}},\dots,v_{j_{k}}), \\
  W_{J} &\to V^{k}\times \UU , &\qquad v &\mapsto (v_{j_{1}},\dots,v_{j_{k}},v_{i}),
\end{alignat}
where \(i\) is some index not in~\(J\).

\smallskip

Define 
\begin{equation}
  \label{eq:def-f}
  f\colon \Vr\setminus X \to\R,
  \quad
  (u,z)\mapsto -\bigl\|u_{1}+\dots+u_{r}\bigr\|^{2}.
\end{equation}

\begin{lemma}
  \label{thm:f-morse}
  This~\(f\) is a
  Morse--Bott function.
  Its critical submanifolds are the~\(P_{J}\) for short~\(J\).
  The negative normal bundle of~\(P_{J}\) in~\(\Vr\setminus X\)
  is its normal bundle in~\(W_{J}\), and the index of~\(P_{J}\) is \( |J|\dV\).
\end{lemma}

\begin{proof}
  Let \((v,w)\in\C^{\aa+\bb}\) be a tangent vector at~\((u,z)\in \Vr\setminus X\). Then
  \begin{equation}
    \label{eq:Df}
    Df(u,z)\cdot(v,w) = 2\pair{u_{1}+\dots+u_{r},v_{1}+\dots+v_{r}}.
  \end{equation}
  If \(z_{j}\ne0\) for some~\(j\), then the map~\((v,w)\mapsto v_{1}+\dots+v_{r}\)
  is surjective. Since \(u_{1}+\dots+u_{r}\ne0\), this implies that \((u,z)\)
  cannot be critical. Hence all critical points satisfy \(z=0\).
  Formula~\eqref{eq:Df} shows that the critical points there are those
  of the restriction of~\(f\) to~\(\UU ^{r}\setminus E_{2\aa}(\ell)\).
  They have been found in~\cite[Lemma~4.3]{FarberFromm:2013}
  to be the submanifolds~\(P_{J}\).

  The restriction of~\(f\) to~\(V_{J}\) assumes its minimum
  at the unique intersection point~\(u\) with~\(P_{J}\), 
  and it is elementary to check that this minimum is non-degenerate.
  Likewise, the restriction of~\(f\) to~\(V_{J^{c}}\) assumes its maximum
  at the unique intersection point with~\(P_{J}\), which is again~\(u\), and this maximum
  is also non-degenerate.
  Since the tangent space of~\(\Vr\) at~\(u\) is the direct sum of the tangent spaces
  of~\(V_{J}\) and \(V_{J^{c}}\), this implies that the Hessian of~\(f\) at~\(u\)
  is non-degenerate with index~\(\dim V_{J}=|J|d\).
  
  By varying the base point~\(*\in \UU \), we can reach any point~\(u\in P_{J}\).
  Hence the claim is proven.
\end{proof}

\begin{lemma}
  \label{thm:f-perfect}
  The Morse--Bott function~\(f\)
  is perfect.
  Moreover, \(H_{*}(\Vr \setminus X;\Z)\) is free with basis
  given by the~\([V_{J}]\) and \([W_{J}]\) for short~\(J\).
\end{lemma}

\begin{proof}
  Let \(g_{1}\),~\ldots,~\(g_{r}\) be generators of the group~\(G=(\Z_{2})^{r}\).
  By letting \(g_{j}\) act as complex conjugation on some coordinate of~\(z_{j}\),
  we get a \(G\)-action on~\(Y=\Vr \setminus X\) such that \(f\) is \(G\)-invariant.
  For any short~\(J\) we have
  \begin{equation}
    g_{j}[V_{J}] = \begin{cases}
      -[V_{J}] & \text{if \(j\in J\),} \\
      +[V_{J}] & \text{if \(j\notin J\),}
      \end{cases}
  \end{equation}
  and analogously for~\([W_{J}]\).

  Let \(c_{1}<\dots<c_{m}<0\) be the critical values of~\(f\),
  and for some small~\(\epsilon\)
  set \(Z_{k}=f^{-1}((-\infty,c_{k}+\epsilon))\subset\Vr\setminus X\).
  Also let~\(Z_{-1}=\emptyset\).
  We prove by induction on~\(k\)
  that a basis of~\(H_{*}(Z_{k})\) is given by the short~\([V_{J}]\) and \([W_{J}]\)
  such that \(f(P_{J})\le c_{k}\).
  (In this proof, all homology is with integer coefficients.)
  For~\(k=m\) this is the desired result.

  Consider the long exact sequence
  \begin{equation}
    \label{eq:seq-Zc-Zcc}
    \cdots \to H_{*}(Z_{k-1}) \to H_{*}(Z_{k}) \to H_{*}(Z_{k},Z_{k-1})
    \stackrel{\delta}{\to} H_{*-1}(Z_{k-1}) \to \cdots \; .
  \end{equation}
  By induction, the~\([V_{J}]\) and \([W_{J}]\) such that \(f(P_{J})<c_{k}\)
  form a basis for~\(H_{*}(Z_{k-1})\).
  Let \(D_{-}\) be union of the negative normal bundles to the~\(P_{J}\) with~\(f(P_{J})=c_{k}\),
  and let \(S_{-}\) be the union of the associated sphere bundles. Then \(H_{*}(Z_{k},Z_{k-1})\cong H_{*}(D_{-},S_{-})\).
  By Lemma~\ref{thm:f-morse}, the images of the~\([V_{J}]\) and \([W_{J}]\) with~\(f(P_{J})=c_{k}\)
  form a basis of these relative homology groups. (Recall that each \(P_{J}\) is a sphere.)
  The discussion of the \(G\)-action above implies that
  \(H_{*}(Z_{k},Z_{k-1})\) and~\(H_{*}(Z_{k-1})\) have no \(G\)-characters in common.
  The map~\(\delta\) in~\eqref{eq:seq-Zc-Zcc} therefore is trivial, and the sequence splits.
  This completes the inductive step.
\end{proof}

\begin{proposition}
  \label{thm:betti-Xell}
  \(H^{*}(X_{\aa,\bb}(\ell);\Z)\) is free, and its Poincaré polynomial is given by
  \begin{equation*}
    \PP(X_{\aa,\bb}(\ell),x) = \sum_{\text{\rm\(J\) short}} x^{|J|\dV} + \sum_{\text{\rm\(J\) long}} x^{|J|\dV-\dU-1}.
  \end{equation*}
  In particular, the Betti sum of~\(X_{\aa,\bb}(\ell)\) is \(2^{r}\).
\end{proposition}

\begin{proof}
  We have \(H^{*}(X;\Z)\cong H_{rd-*}(\Vr,\Vr \setminus X;\Z)\)
  by Poincaré--Alexander--Lef\-schetz duality.
  Thus, it is enough to verify that this relative homology
  is free and with Poincaré polynomial
  \begin{equation*}
    \sum_{\text{\(J\) long}} x^{|J|\dV} + \sum_{\text{\(J\) short}} x^{|J|\dV+\dU+1}.
  \end{equation*}
  (Recall that long subsets and short subsets are complements of each other
  and that the dimension of~\(\Vr\) is~\(rd\).)

  Let \(\iota_{*}\colon H_{*}(\Vr\setminus X;\Z)\to H_{*}(\Vr;\Z)\)
  be the map induced by the inclusion.
  Then for short~\(J\)
  \begin{equation}
    \label{eq:iota-VJ-WJ}
    \iota_{*}[V_{J}] = [V_{J}]
    \qquad\text{and}\qquad
    \iota_{*}[W_{J}] = 0.
  \end{equation}
  For the second identity it is enough to consider the case~\(J=\emptyset\),
  where it is true for degree reasons.
  Our claim follows from~\eqref{eq:iota-VJ-WJ} and the short exact sequence
  \begin{equation}
    0 \to \coker\iota_{*} \to H_{*}(\Vr,\Vr\setminus X;\Z) \to \ker\iota_{*-1} \to 0.
    \qedhere
  \end{equation}
\end{proof}

\begin{remark}
  \label{rem:bigchain}
  The manifold~\(\BC_{\bb+1}^{r+1}(\ell,0)\)
  defined in~\cite[\S 10.3.1]{Hausmann:mod2}%
  \footnote{Strictly speaking, the definition of a big chain space 
  in~\cite{Hausmann:mod2} requires all edges to have positive length. However,
  as in the proof of Lemma~\ref{thm:X-ell-mf}\,\eqref{thm:X-ell-mf-3} 
  one can see
  that \(\smash{\BC_{b+1}^{r+1}(\ell,0)}\) is diffeomorphic to the big chain space~\(\smash{\BC_{b+1}^{r+1}(\ell,\epsilon)}\)
  for small~\(\epsilon\).}
  is the fixed point set of the involution of~\(X_{1,\bb}(\ell)\) induced
  by the complex conjugation on~\(\C^{1+\bb}\),
  hence a sort of ``real locus'' of~\(X_{1,\bb}(\ell)\).
  We note that the mod~\(2\) Betti numbers of~\(\BC_{\bb+1}^{r+1}(\ell,0)\) computed
  in~\cite[Thm.~10.3.16]{Hausmann:mod2} are the same as those of~\(X_{1,\bb}(\ell)\),
  up to degree shifts.  
\end{remark}

\begin{remark}
  \label{rem:betti-polygon}
  The Betti numbers of the spaces of polygons~\(E_{2\aa}(\ell)\)
  have been computed by Farber--Schütz~\cite[Thm.~1]{FarberSchuetz:2007} (\(\aa=1\))
  and Fromm~\cite[Sec.~3.1]{Fromm:2011} (\(\aa\ge2\)).
  In contrast to big polygon spaces, the Betti sum of a polygon space depends on the length vector:
  The Betti sum of~\(E_{2\aa}(\ell)\) is four times the number of
  short subsets~\(J\subset[r]\) containing the largest index~\(r\).
  As shown in~\cite[Thm.~2]{FarberSchuetz:2007}, there is the sharp upper bound
  \begin{equation}
    \label{eq:betti-polygon-bound}
    \dim H^{*}(E_{2\aa}(\ell)) \le 2^{r} - 2\binom{2m}{m}
  \end{equation}
  if \(r=2m+1\) is odd and twice this number if \(r=2m+2\) is even.
  The maximum is realized if and only if
  \(\ell\sim(1,\dots,1)\)
  (odd~\(r\)) or
  \(\ell\sim(0,1\dots,1)\) (even~\(r\)).
  For example, for~\(r=2m+1\) and \(\ell=(1,\dots,1)\) the Poincaré polynomial
  is
  \begin{equation}
    \label{eq:poincare-polygon}
    \PP(E_{2\aa}(\ell),x) =
    \sum_{j=0}^{m-1} \binom{r}{j} \bigl(x^{j\dU}+x^{(2m-j)\dU-1}\bigr)
    + \binom{2m}{m-1}\bigl(x^{m\dU}+x^{m\dU-1}\bigr).
  \end{equation}
\end{remark}

We now describe the product structure of~\(H^{*}(X;\Z)\).

\begin{proposition}
  \label{thm:product-HX-full}
  There is a basis of~\(H^{*}(X_{\aa,\bb}(\ell);\Z)\)
  consisting of elements~\(\alpha_{J}\) of degree~\(|J|\dV\) (\(J\)~short)
  and elements~\(\beta_{J}\) of degree~\(|J|\dV-\dU-1\) (\(J\)~long) such that
  \(\alpha_{\emptyset}=1\) and
  \begin{align*}
    \alpha_{J}\cup\alpha_{K} &=
    \begin{cases}
      \shufflesign{J}{K}\,\alpha_{J\cup K} & \text{if \(J\cap K=\emptyset\) and \(J\cup K\)~short,} \\
      0 & \text{otherwise,}
    \end{cases} \\
    \alpha_{J}\cup\beta_{K} &=
    \begin{cases}
      \shufflesign{J}{K}\,\beta_{J\cup K} & \text{if \(J\cap K=\emptyset\),} \\
      0 & \text{otherwise,}
    \end{cases} \\
    \beta_{J}\cup\beta_{K} &= 0.
  \end{align*}
  Moreover, \(H^{d*}(X_{\aa,\bb}(\ell);\Z)=\langle\,\alpha_{J}\mid\text{\(J\)~short}\,\rangle\) is
  the image of the restriction map\linebreak \(H^{*}(\Vr;\Z)\to H^{*}(X_{\aa,\bb}(\ell);\Z)\).
\end{proposition}

Consequently, \(H^{d}(X)\) generates a subalgebra of~\(H^{*}(X)\) of dimension~\(2^{r-1}\).

\begin{proof}
  We combine Fromm's approach to the ring structure
  of the cohomology of spaces of polygons \cite[Prop.~A.2.4]{Fromm:2011}
  with representation theory.
  As in the proof of Lemma~\ref{thm:f-perfect},
  we use the action of~\(G=(\Z_{2})^{r}\) on the various spaces.
  The characters of~\(G\) are canonically indexed by the subsets of~\([r]\)
  with \(\emptyset\) corresponding to the trivial character.

  Each~\([V_{J}]\) and \([W_{J}]\) transforms according to the character~\(J\).
  From the proof of Proposition~\ref{thm:betti-Xell} we see that
  each character~\(J\subset[r]\) occurs
  in~\(H_{*}(\Vr,\Vr \setminus X;\Z)\) with multiplicity~\(1\);
  the corresponding isotypical component is spanned by the image
  of~\([V_{J}]\) if \(J\) is long and by a preimage of~\([W_{J}]\)
  if \(J\) is short.

  From the naturality of Poincaré--Alexander--Lefschetz duality
  we get the following commutative diagram:
  \begin{equation}
    \begin{tikzcd}
      H^{*}(\Vr ;\Z) \arrow{d}{\cong} \arrow{r} & H^{*}(X;\Z) \arrow{d}{\cong} \\
      H_{rd-*}(\Vr ;\Z) \arrow{r} & H_{rd-*}(\Vr,\Vr \setminus X;\Z)  
    \end{tikzcd}
  \end{equation}
  The vertical isomorphisms are induced by the cap product with the
  fundamental class~\([V_{[r]}]\) of~\(\Vr\). This implies that they interchange
  the isotypical components corresponding to~\(J\) and~\(J^{c}\).

  We define the~\(\alpha_{J}\)'s as the duals
  of the~\(V_{J^{c}}\). In particular, \(\alpha_{\emptyset}=1\).
  Since the~\(\alpha_{J}\)'s are images of the corresponding elements
  in~\(H^{*}(\Vr)\), we get their multiplication rule as well as
  the last claim. The \(\beta_{J}\)'s are duals of the preimages of the~\([W_{J^{c}}]\)'s.
  By Schur's lemma and Poincaré duality, we have
  \(\alpha_{J^{c}}\cup\beta_{J}=\pm\beta_{[r]}\); we choose \(\beta_{J}\)
  such that the sign equals \(\sigma_{J^{c},J}\).

  Now assume that \(J\) is short and \(K\)~long. If \(J\) and \(K\) are disjoint,
  then \(\alpha_{J}\cup\beta_{K}\) must be a multiple of~\(\beta_{J\cup K}\),
  again by Schur's lemma.
  To see that the scalar is as claimed, we set \(L=(J\cup K)^{c}\)
  and compute
  \begin{align}
    \alpha_{L}\cup(\alpha_{J}\cup\beta_{K})
    &= \sigma_{L,J}\,\alpha_{L\cup J}\cup\beta_{K}
    = \sigma_{L,J}\,\sigma_{L\cup J,K}\,\beta_{[r]} \\
    &= \sigma_{L,J\cup K}\,\shufflesign{J}{K}\,\beta_{[r]}
    = \alpha_{L}\cup(\shufflesign{J}{K}\,\beta_{J\cup K}).
  \end{align}
  If \(J\) and \(K\) are not disjoint, then \(\alpha_{J}\cup\beta_{K}\)
  is a multiple of~\(\alpha_{J\bigtriangleup K}\)
  or~\(\beta_{J\bigtriangleup K}\),
  depending on whether the symmetric difference~\(J\bigtriangleup K\)
  is short or long. But the degree of either candidate
  is strictly smaller than the sum of the degrees
  of~\(\alpha_{J}\) and \(\beta_{K}\). Hence the product vanishes.

  The degree of an~\(\alpha_{J}\) is congruent to~\(0\) modulo~\(\dV\)
  and that of a~\(\beta_{J}\) congruent to~\(-\dU-1\equiv 2b\).
  Hence the degree of a product~\(\beta_{J}\cup\beta_{K}\)
  is congruent to~\(4b\) modulo~\(d\). Since \(\dV\) is odd and \(b\not\equiv0\),
  \(4b\) is neither congruent to~\(0\) nor to~\(2b\), which implies
  that such a product vanishes, too.  
\end{proof}

The Walker conjecture~(1985) asserted that two generic length vectors~\(\ell\),~\(\ell'\)
are equivalent if \(E_{2}(\ell)/SO(2)\) and \(E_{2}(\ell')/SO(2)\) have isomorphic integral cohomology rings.
This was finally proven by Schütz in~2010, based on work of Farber--Hausmann--Schütz;
the analogous question for the spaces of polygons~\(E_{2a}(\ell)\) was resolved by Farber--Fromm,
see~\cite{FarberFromm:2013} and~\cite[\S 10.3.4]{Hausmann:mod2}.
Using Proposition~\ref{thm:product-HX-full}, we can easily obtain
a version for big polygon spaces.

\begin{proposition}
  Let \(\ell\) and~\(\ell'\) be two generic length vectors.
  Then \(\ell\sim\ell'\) if and only if
  \(H^{*}(X_{\aa,\bb}(\ell);\Z_{2})\) and \(H^{*}(X_{\aa,\bb}(\ell');\Z_{2})\)
  are isomorphic as graded rings.
\end{proposition}

\begin{proof}
  We only have to do the `if' part;
  the `only if' is Lemma~\ref{thm:X-ell-mf}\,\eqref{thm:X-ell-mf-3}.
  
  Since \(H^{*}(X_{\aa,\bb}(\ell);\Z)\) is torsion-free
  by Proposition~\ref{thm:betti-Xell} (or Proposition~\ref{thm:product-HX-full}), we have
  \(H^{d*}(X(\ell);\Z_{2})=H^{d*}(X_{\aa,\bb}(\ell);\Z)\otimes_{\Z}\Z_{2}\);
  we denote this ring by~\(A^{*}(\ell)\).
  By assumption, \(A^{*}(\ell)\cong A^{*}(\ell')\) as graded rings.
  
  The \(\Z_{2}\)-dimension of~\(A^{1}(\ell)\) is the number of \(\ell\)-short singleton sets.
  It is either \(r\) or~\(r-1\) as two long sets always intersect. Given that
  we assume length vectors to be weakly increasing, the latter case implies
  that any set containing~\(r\) is long. These are already half of all subsets,
  so all other subsets are short. Hence \(\ell\sim(0,\dots,0,1)\) in this case.
  Because the same applies to~\(\ell'\), we may assume all singleton sets
  to be \(\ell\)-short and \(\ell'\)-short.

  From Proposition~\ref{thm:product-HX-full} we see that
  \begin{equation}
    A^{*}(\ell) = H^{d*}(\Vr;\Z_{2}) \bigm/ \langle\, \alpha_{J} \mid \text{\(J\subset[r]\) \(\ell\)-long}\,\rangle,
  \end{equation}
  and analogously for~\(A^{*}(\ell')\). From a result of Gubeladze
  on isomorphic monoid rings (\cf~\cite[Thm.~4.7.53]{Hausmann:mod2})
  it now follows that \(\ell\) and~\(\ell'\) define the same notion of
  `long' and `short'.
  (We have excluded the case of long singleton sets to ensure
  that the abstract simplicial complexes on~\([r]\) defined by the
  \(\ell\)-short and \(\ell'\)-short subsets indeed have \(r\)~vertices,
  as required by the definition in~\cite[\S 2.1]{Hausmann:mod2}.)
\end{proof}

\section{Equivariant cohomology}
\label{sec:proof-main}

It will be convenient to consider equivariant homology 
along with equivariant cohomology. We therefore start with some general remarks about the former;
details can be found in~\cite[Sec.~3]{AlldayFranzPuppe:orbits1} or~\cite[Sec.~2]{AlldayFranzPuppe:orbits4}.
We remind the reader that this equivariant homology is \emph{not} the homology of the Borel construction.

Let \(X\) be a \(T\)-manifold of dimension~\(n\).
The equivariant homology~\(\smash{\hHT_{*}(X)}\) of~\(X\) (with compact supports)
as well as the equivariant homology~\(\smash{\hHTc_{*}(X)}\) with closed supports
are modules over the polynomial ring~\(\RR =\kk[\ttt_{1},\dots,\ttt_{r}]\); for compact~\(X\) they coincide.

There is a canonical restriction map
\begin{equation}
  \label{eq:restriction-hHT}
  \hHTc_{*}(X)\to \hHc_{*}(X)=\Hom(\Hc^{*}(X),\kk) \, ;
\end{equation}
it is the edge homomorphism of a spectral sequence with~%
\( 
  E_{2} = \hHc_{*}(X)\otimes\RR 
  \) 
and converging to~\(\smash{\hHTc_{*}(X)}\),
see~\cite[Prop.~2.3]{AlldayFranzPuppe:orbits4}.
Under this map,
any orientation~\(o\in\hHc_{n}(X)\) of~\(X\) lifts uniquely to an equivariant orientation \(o_{T}\in\hHTc_{n}(X)\)
\cite[Prop.~3.2]{AlldayFranzPuppe:orbits4}.
The equivariant Poincaré duality isomorphism~\(\PD_{X}\colon\HT^{*}(X)\to\hHTc_{n-*}(X)\)
is the cap product with~\(o_{T}\).

Suppose that \(X\) is a \(T\)-stable closed submanifold of a \(T\)-manifold~\(Y\) and
let \(\nu^{T}_{*}\colon\hHTc_{*}(X)\to\hHTc_{*}(Y)\) be the map induced by the inclusion.
The orientation of~\(X\) being understood, we write \([X]_{T}\in\hHTc_{*}(Y)\)
for its image under~\(\nu^{T}_{*}\).

\begin{proposition}
  \label{thm:product-KL}
  Let \(T=K\times L\) be a decomposition into subtori, inducing a decomposition
  \(\RR=\RR_{K}\otimes\RR_{L}\) (with the obvious meaning).
  For a \(K\)-manifold~\(X\) and an \(L\)-manifold~\(Y\) there is
  an isomorphism of~\(\RR\)-modules
  \begin{equation*}
    \mathord{\times}\colon H^{K,c}_{*}(X)\otimes H^{L,c}_{*}(Y)\to \hHTc_{*}(X\times Y).
  \end{equation*}
  For any \(K\)-stable oriented closed submanifold~\(M\subset X\) and any \(L\)-stable oriented closed submanifold~\(N\subset Y\)
  one has
  \( 
    [M]_{K}\times[N]_{L}= [M\times N]_{T} 
  \). 
\end{proposition}

\begin{proof}
  We denote by~\(\CTc^{*}(-)\) the singular Cartan model
  for equivariant cohomology with compact supports;
  its \(\RR\)-dual \(\smash{\hCTc_{*}(-)}=\Hom_{\RR}(\CTc^{*}(-),\RR)\) gives rise
  to equivariant homology with closed supports, see~\cite[Sec.~2.3]{AlldayFranzPuppe:orbits4}.
  Also, let \(\pi_{X}\) and~\(\pi_{Y}\) be the projections of~\(X\times Y\) onto~\(X\) and~\(Y\), respectively.
  
  The well-known cross product isomorphism in equivariant cohomology
  (here with compact supports)
  is induced by the quasi-isomorphism of \(\RR\)-algebras
  \begin{equation}
    \label{eq:OmegaKX-OmegaLY}
    C_{K,c}^{*}(X)\otimes C_{L,c}^{*}(Y) \to \CTc^{*}(X\times Y),
    \quad
    \alpha\otimes\beta\mapsto\pi_{X}^{*}(\alpha)\cup\pi_{Y}^{*}(\beta).
  \end{equation}
  Moreover, because we assume \(H_{c}^{*}(X)\cong H_{*}(X)\) and \(H_{c}^{*}(Y)\cong H_{*}(Y)\) to be finite-dimensional,
  the map
  \begin{equation}
    \Hom_{\kk}(C_{c}^{*}(X),\kk)\otimes\Hom_{\kk}(C_{c}^{*}(Y),\kk)
    \to
    \Hom_{\kk}(C_{c}^{*}(X)\otimes C_{c}^{*}(Y),\kk)
  \end{equation}
  is a quasi-isomorphism of complexes.
  A spectral sequence argument as in~\cite[Rem.~3.3]{AlldayFranzPuppe:orbits1}
  shows that its equivariant extension
  \begin{equation}
    C^{K,c}_{*}(X)\otimes C^{L,c}_{*}(Y) \to 
    \Hom_{\RR}(C_{K,c}^{*}(X)\otimes C_{L,c}^{*}(Y),\RR)
  \end{equation}
  is a quasi-isomorphism of \(\RR\)-modules.
  Combining it with the Künneth formula and the quasi-isomorphism
  dual to~\eqref{eq:OmegaKX-OmegaLY} establishes the isomorphism in equivariant homology.

  The last claim follows by verifying
  that both \([M]_{K}\otimes[N]_{L}\) and \([M\times N]_{T}\)
  restrict to~\([M]\times[N]=[M\times N]\in H^{c}_{*}(M\times N)\)
  according to the way we orient products.
\end{proof}

Let \(X\) be a closed \(T\)-stable submanifold
of a \(T\)-manifold~\(Y\) with both \(X\) and \(Y\) oriented.
The \emph{equivariant Euler class} of~\(X\subset Y\) then is
\(e_{T}(X\subset Y)=\nu_{T}^{*}\,\nu^{T}_{!}(1)\),
where the push-forward map~\(\nu^{T}_{!}\colon\HT^{*}(X)\to\HT^{*}(Y)\)
is defined as the composition~\(\nu^{T}_{!}=\PD_{Y}^{-1}\nu^{T}_{*}\PD_{X}\).
If \(G=S^{1}\) acts by scalar multiplication on~\(\C\) (with the canonical orientation),
then~\(e_{G}(*\subset\C)=\ttt\in\kk[\ttt]=H^{*}(BG)\).
We will need a related case.

\begin{lemma}
  \label{thm:incl-U-V}
  With the above notation,
  let \(G\) act trivially on~\(\C^{\aa}\) and by scalar multiplication on~\(\C^{\bb}\),
  and let \(S\subset\C^{\aa}\times\C^{\bb}\) be the unit sphere. Then
  \begin{equation*}
    e_{G}(S^{G}\subset S)=\ttt^{\bb}.
  \end{equation*}
  Equivalently, \([S^{G}]_{G}=\ttt^{\bb}\cdot[S]_{G}\in H^{G}_{*}(S)\).
\end{lemma}

\begin{proof}
  By naturality,
  we can replace \(S\) by the normal bundle~\(W\) of~\(S^{G}\) in~\(S\). This bundle is trivial,
  \(W\cong S^{G}\times\C^{\bb}\), where \(G\) acts by scalar multiplication on~\(\C^{\bb}\),
  and the product orientation on~\(W\) coincides with the one inherited from~\(S\).
  This implies
  \begin{equation}
    e_{G}(S^{G}\subset S) = e_{G}(*\subset\C^{\bb}) = (e_{G}(*\subset\C))^{\bb} = \ttt^{\bb}.
  \end{equation}

  Clearly, \(\ttt^{\bb}\cdot[S]_{G}\) is Poincaré dual to~\(\ttt^{\bb}\in H_{G}^{*}(S)\).
  So the homological formulation follows once we observe that the restriction map~\(H_{G}^{*}(S)\to H_{G}^{*}(S^{G})\)
  is injective. This can be seen by a direct computation, or as follows:
  Since \(S\) and \(S^{G}\) have the same Betti sum, \(H_{G}^{*}(S)\) is free
  over~\(\kk[\ttt]\), \cf~\cite[Thm.~3.10.4]{AlldayPuppe:1993}.
  Hence restriction to the fixed point set 
  is injective, for example because of the
  Chang--Skjelbred sequence~\eqref{eq:6:chang-skjelbred}.
\end{proof}

\smallskip

We are now ready to look at
the equivariant cohomology of the big polygon space~\(X=X_{\aa,\bb}(\ell)\).
Before starting in earnest, we make a simple observation.
It will be sharpened in Proposition~\ref{thm:ell-bound-syzygy},
that time without appealing to~\cite{FarberSchuetz:2007}.

\begin{lemma}
  \label{thm:HTX-not-free}
  \(\HT^{*}(X)\) is not free over~\(\RR\). In fact,
  \(\syzord\HT^{*}(X)<r/2\).
\end{lemma}

\begin{proof}
  By comparing the Farber--Schütz bound~\eqref{eq:betti-polygon-bound} with Proposition~\ref{thm:betti-Xell},
  we see that the Betti sum of~\(X\) is always larger
  than that of its fixed point set.
  Similar to the preceding proof,
  this implies that
  \(\HT^{*}(X)\) is not free over~\(\RR\).
  The latter claim now follows from Theorem~\ref{thm:syzygy-PD-space}.
\end{proof}

Let \(\iota^{T}_{*}\colon\hHT_{*}(\Vr \setminus X)\to\hHT_{*}(\Vr )\) be the
equivariant analogue of~\(\iota_{*}\). 

\begin{lemma}
  \label{thm:extension-HTX}
  There is a short exact sequence\footnote{%
  By an argument due to Puppe~\cite[Lemma~3.12]{Puppe:2018},
  one can show that this sequence splits.}
  \begin{equation*}
    0 \to (\coker\iota^{T}_{*})[rd] \to \HT^{*}(X) \to (\ker\iota^{T}_{*})[rd-1] \to 0.
  \end{equation*}
\end{lemma}

Note that here and throughout we use a cohomological grading. For example,
an element~\(c\in\hHT_{k}(\Vr)\) has degree~\(-k\), and degree~\(rd-k\) in~\(\hHT_{*}(\Vr)[rd]\).

\begin{proof}
  Set \(n=\dim \Vr =r\dV\).
  Because of the naturality of
  equivariant Poincaré--Alexander--Lefschetz duality~\cite[Thm.~3.4]{AlldayFranzPuppe:orbits4},
  the following diagram is commutative:
  \begin{equation}
    \begin{tikzcd}[font=\small,column sep=small] 
      \HT^{n-*-1}(X) \arrow{d}{\cong} \arrow{r} & \HT^{n-*}(\Vr ,X) \arrow{d}{\cong} \arrow{r} & \HT^{n-*}(\Vr ) \arrow{d}{\cong} \arrow{r} & \HT^{n-*}(X) \arrow{d}{\cong} \\
       \hHT_{*+1}(\Vr,\Vr \setminus X) \arrow{r} & \hHT_{*}(\Vr \setminus X) \arrow{r}{\iota^{T}_{*}} & \hHT_{*}(\Vr ) \arrow{r} & \hHT_{*}(\Vr,\Vr \setminus X)  \mathrlap{.}
    \end{tikzcd}
    \qedhere 
  \end{equation}
\end{proof}

\begin{lemma} \( \)
  \label{thm:hHT-free}
  \begin{enumerate}
  \item
    \label{thm:hHT-Vr-free}
    \(\hHT_{*}(\Vr )\) is a free \(\RR \)-module with basis \([V_{J}]_{T}\), \(J\subset[r]\).
  \item
    \label{thm:f-equiv-perfect}
    \(\hHT_{*}(\Vr \setminus X)\) is a free \(\RR \)-module with basis
    \([V_{J}]_{T}\) and \([W_{J}]_{T}\), \(J\)~short.
  \end{enumerate}
\end{lemma}

\begin{proof}
  The~\([V_{J}]_{T}\) are preimages of the~\([V_{J}]\)
  under the restriction map~\eqref{eq:restriction-hHT}.
  Since the latter form a basis of~\(H_{*}(\Vr)\), this implies,
  as in the Leray--Hirsch theorem, that the
  spectral sequence
  \(E_{2} = \hHc_{*}(\Vr)\otimes\RR \Rightarrow\hHT_{*}(\Vr)\)
  collapses 
  and that the \([V_{J}]_{T}\) form a basis of~\(\hHT_{*}(\Vr)\) over~\(\RR \).

  By Lemma~\ref{thm:f-perfect},
  a basis for~\(H_{*}(\Vr \setminus X)\) is given by the~\([V_{J}]\) and \([W_{J}]\) for short~\(J\).
  Hence the same proof works for~\(\hHT_{*}(\Vr \setminus X)\).
\end{proof}

\begin{proposition}
  \label{thm:iota-V-W}
  For \(J\)~short,
  \begin{align*}
    \iota^{T}_{*}[V_{J}]_{T} &= [V_{J}]_{T}, \\
    \iota^{T}_{*}[W_{J}]_{T} &= \sum_{j\notin J} \shufflesign{J}{j} \, \ttt_{j}^{\bb}\cdot [V_{J\cup j}]_{T}.
  \end{align*}
\end{proposition}

\begin{proof}
  The claim is clear for~\(V_{J}\).
  For~\(W_{J}\), consider first the case~\(J=\emptyset\),
  where \(W_{\emptyset}=\Delta_{J}\) is the diagonal of~\(\Ur\). Then
  \begin{equation}
    [\Delta_{J}] = \sum_{j\in[r]} [\UU^{\{j\}}\times \{*\}^{[r]\setminus j}]
  \end{equation}
  in~\(H_{*}(\Ur)\), hence also in~\(\hHT_{*}(\Ur)=H_{*}(\Ur)\otimes\RR\).
  By naturality and Lemma~\ref{thm:incl-U-V}, we get
  \begin{equation}
    \iota^{T}_{*}[W_{\emptyset}]_{T} = \sum_{j\in[r]} \ttt_{j}^{\bb}\cdot[V_{\{j\}}]_{T}
    \in \hHT_{*}(\Vr).
  \end{equation}
  The general case now follows from Proposition~\ref{thm:product-KL}:
  Define the subtori~\(K=(S^{1})^{J}\) and \(L=(S^{1})^{I}\) of~\(T=(S^{1})^{r}\), where \(I=J^{c}\),
  and let \(\Delta_{I}\subset\UU^{I}\) be the diagonal and \(s\) the sign of the shuffle~\((J,I)\).
  Then
  \begin{equation}
  \begin{split}
    \iota^{T}_{*}[W_{J}]_{T} &= s\,\iota^{T}_{*}([V_{J}]_{K}\times[\Delta_{I}]_{L}) \\
    &= s \sum_{j\in I} [V_{J}]_{K}\times (\ttt_{j}^{\bb}\cdot[V_{\{j\}}]_{L})
    = \sum_{j\notin J} \shufflesign{J}{j} \, \ttt_{j}^{\bb}\cdot[V_{J\cup j}]_{T}
  \end{split}    
  \end{equation}
  because of the way the orientation of each~\(V_{J\cup j}\) is defined.
\end{proof}

\section{The equilateral case}
\label{sec:equilateral}

We now consider the equilateral case given by~\(\ell=(1,\dots,1)\in\R^{r}\).
This length vector is generic if and only if \(r=2m+1\) is odd. In this case,
a subset~\(J\subset[r]\) is short if and only if \(|J|\le m\).
We are going to  identify \(\HT^{*}(X_{\aa,\bb}(\ell))\) with the syzygies
appearing in the Koszul resolution of~\(\RR/(\ttt_{1}^{\bb},\dots,\ttt_{r}^{\bb})\),
which we review first.

Let \(N\) be an \(r\)-dimensional \(\kk\)-vector space,
concentrated in degree~\(2\bb\), and let
\((e_{1},\dots,e_{r})\) be a basis of the
the \(\kk\)-dual~\(\tilde N\) of~\(N\).
(Recall that the generators of~\(\RR=\kk[\ttt_{1},\dots,\ttt_{r}]\) have degree~\(2\).)
We write \(N^{\wedge k}\) for the \(k\)-th exterior power of~\(N\).
The Koszul resolution of~\(M=\RR/(\ttt_{1}^{\bb},\dots,\ttt_{r}^{\bb})\) over~\(\RR\) is
\begin{multline}
  \label{eq:koszul-res}
  0 \stackrel{\delta_{r+1}}{\longrightarrow} \RR \otimes N^{\wedge r}
  \stackrel{\delta_{r}}{\longrightarrow} \RR \otimes N^{\wedge(r-1)}
  \stackrel{\delta_{r-1}}{\longrightarrow} \dots \\
  \stackrel{\delta_{3}}{\longrightarrow} \RR \otimes N^{\wedge 2}
  \stackrel{\delta_{2}}{\longrightarrow} \RR \otimes N
  \stackrel{\delta_{1}}{\longrightarrow} \RR
  \stackrel{\delta_{0}}{\longrightarrow} M
  \longrightarrow 0
\end{multline}
with 
connecting homomorphisms
\begin{equation}
  \delta_{k}\colon \RR \otimes N^{\wedge k} \to \RR \otimes N^{\wedge(k-1)},
  \quad
  f\otimes \alpha \mapsto \sum_{j=1}^{r} f\ttt_{j}^{\bb}\otimes e_{j}\contr\alpha
\end{equation}
for~\(k>0\);
here \(e_{j}\contr\alpha\) denotes the contraction of~\(\alpha\) with~\(e_{j}\).
The map~\(\delta_{0}\) is the canonical projection.

For~\(0\le k\le r+1\), we define the \emph{\(k\)-th Koszul syzygy} to be
\begin{equation}
  \KK{k} 
  = \im\delta_{k}[-2\bb k],
\end{equation}
\cf~\cite[Sec.~2.4]{AlldayFranzPuppe:orbits1};
the degree shift ensures that each~\(\KK{k}\) is generated in degree~\(0\).
For example, \(\KK{0}=M\),
\(\KK{1}[2\bb]=(\ttt_{1}^{\bb},\dots,\ttt_{r}^{\bb})\lhd\RR\)
(which for~\(\bb=1\) is the maximal homogeneous ideal),
\(\KK{r}=\RR \) and
\(\KK{r+1}=0\).
It is clear from the definition that \(\KK{k}\) is a \(k\)-th syzygy.
In fact, for~\(k\le r\) we have
\begin{equation}
  \label{eq:syzord-Kk}
  \syzord \KK{k} = k
\end{equation}
because otherwise
Hilbert's syzygy theorem would imply that
the homological dimension of~\(\KK{k}\) is less than~\(r-k\)
and therefore that of~\(M\) less than~\(r\).
But this is impossible
as setting all~\(\ttt_{j}=0\) in the resolution~\eqref{eq:koszul-res}
gives \(\Tor^{\RR}_{r}(M,\kk)=\kk[2\bb r]\).

We write the basis of~\(\tilde N^{\wedge k}\) induced by the chosen basis of~\(\tilde N\) as
\begin{equation}
  e_{J}=e_{j_{1}}\wedge\dots\wedge e_{j_{k}}
  \quad
  \text{for}
  \quad
  J=\{j_{1}<\dots<j_{k}\} \subset [r].
\end{equation}
Note that \(e_{J}\) is of degree~\(-2b|J|\).
Because of the self-duality of the resolution~\eqref{eq:koszul-res},
\(\KK{k+1}[-2\bb(r-k-1)]\) and \(\KK{k}[-2\bb(r-k)]\) are for~\(1\le k\le r\) the kernel and image, respectively, of the map
\begin{align}
  \tilde\delta_{r-k}\colon \RR \otimes \tilde N^{\wedge(r-k)} &\to \RR \otimes \tilde N^{\wedge(r-k+1)}, \\
  \label{eq:koszul-diff-homology}
  f\otimes e_{J} &\mapsto \sum_{j=1}^{r} f\ttt_{j}^{\bb}\otimes e_{J}\wedge e_{j}
  = \sum_{j\notin J} \shufflesign{J}{j} \, f\ttt_{j}^{\bb}\otimes e_{J\cup j}.
\end{align}

\begin{proposition}
  \label{thm:HTX-equilateral}
  Let \(r=2m+1\) and \(\ell=(1,\dots,1)\in\R^{r}\). Then
  \begin{align*}
    \HT^{*}(X_{\aa,\bb}(\ell)) &\cong \!\bigoplus_{|J|<m}\!\! \RR\bigl[|J|\dV\mspace{1mu}\bigr]
    \oplus \KK{m}[m\dV\mspace{1mu}] \\
    &\quad \oplus \KK{m+2}\bigl[(m+2)\dV-2\dU-1\bigr]
    \oplus \!\!\!\bigoplus_{|J|>m+1}\!\!\!\! \RR\bigl[|J|\dV-\dU-1\bigr].
  \end{align*}
  In particular,
  \(\syzord\HT^{*}(X_{\aa,\bb}(\ell))=m\).
\end{proposition}

\begin{proof}
  We start by computing the kernel and cokernel of~\(\iota^{T}_{*}\) in the short exact sequence
  \begin{equation}
    \label{eq:extension-HTX}
    0 \to (\coker\iota^{T}_{*})[rd] \to \HT^{*}(X) \to (\ker\iota^{T}_{*})[rd-1] \to 0
  \end{equation}
  from Lemma~\ref{thm:extension-HTX}.
  It follows from Proposition~\ref{thm:iota-V-W}
  that for any~\(0\le j\le m\) the restriction
  \begin{equation}
    \iota^{T}_{*}\colon \bigoplus_{|J|=j} R\,[W_{J}]_{T} \to \!\! \bigoplus_{|J|=j+1} \!\! R\,[V_{J}]_{T}
  \end{equation}
  is essentially the
  map~\(\tilde\delta_{j}\) from~\eqref{eq:koszul-diff-homology}, up to a degree shift by~\(|W_{J}|-|e_{J}|=-(j+1)\dU\),
  compare \eqref{eq:dim-VJ-WJ}. (Note that we grade homology negatively.) Set \(\bar s=-(m+1)\dU\) and \(s=\bar s-2b(m+1)=-(m+1)\dV\).
  The kernel of~\(\iota^{T}_{*}\) is spanned by the elements
  \begin{equation}
    [W_{J}]_{T} - \sum_{j\notin J} \shufflesign{J}{j} \, \ttt_{j}^{\bb}\cdot [V_{J\cup j}]_{T}
  \end{equation}
  for~\(|J|<m\) plus the kernel~\(\KK{m+2}[\bar s-2b(m-1)]=\KK{m+2}[-(m-1)\dV-2\dU]\)
  of \(\tilde\delta_{m}[\bar s]=\tilde\delta_{r-m-1}[\bar s]\),
  and the cokernel of~\(\iota^{T}_{*}\) is spanned by the~\([V_{J}]\) with~\(|J|>m+1\)
  plus the cokernel of~\(\tilde\delta_{r-m-1}[\bar s]\),
  which is the image~\(\KK{m}[\bar s-2b(m+1)]\) of~\(\tilde\delta_{r-m}[\bar s]\).
  Hence
  \begin{align}
    \ker\iota^{T}_{*} &\cong \bigoplus_{|J|<m} \RR \bigl[-|J|\dV-\dU\,\bigr] \oplus \KK{m+2}\bigl[-(m-1)\dV-2\dU\,\bigr], \\
    \coker\iota^{T}_{*} &\cong  \KK{m}\bigl[- (m+1)\dV\,\bigr] \oplus \!\!\bigoplus_{|J|>m+1}\!\! \RR \bigl[- |J|\dV\,\bigr].
  \end{align}

  Next we show that
  the extension~\eqref{eq:extension-HTX} is trivial.
  The free summands of~\(\ker\iota^{T}_{*}\) clearly pose no problem.
  Because \(\KK{m}\) and \(\KK{m+2}\) both live in even degrees
  and their degree shifts in~\eqref{eq:extension-HTX} differ by an odd number,
  there is no problem, either. Finally, extensions of the form
  \begin{equation}
    0 \to \RR [l] \to M \to \KK{j}[l']\to 0
  \end{equation}
  are always trivial if \(l-l'\ne2\bb\).
  For~\(\bb=1\) this has been shown in~\cite[Lemma~2.4]{AlldayFranzPuppe:orbits1};
  the general case is analogous.
  The sequence~\eqref{eq:extension-HTX} thus splits.

  Hence \(\HT^{*}(X)\) is a direct sum of \(m\)-th syzygies by~\eqref{eq:syzord-Kk},
  so it is an \(m\)-th syzygy itself.
  This is the maximum possible by Lemma~\ref{thm:HTX-not-free}.
\end{proof}

\begin{lemma}
  \label{rem:even-r}
  Let \(K\),~\(L\) be tori and set \(T=K\times L\).
  Let \(Y\) be a \(K\)-manifold such that
  \(H_{K}^{*}(Y)\) is not free over~\(\RR_{K}\),
  and let \(Z\) be an \(L\)-manifold
  such that \(H_{L}^{*}(Z)\)
  is free over~\(\RR_{L}\).
  Then
  \begin{equation*}
    \syzord_{\RR} \HT^{*}(Y\times Z)=\syzord_{\RR_{K}}H_{K}^{*}(Y).
  \end{equation*}
\end{lemma}

\begin{proof}
  Let \(j=\syzord H_{K}^{*}(Y)<r\).
  Tensoring an exact sequence of the form~\eqref{eq:def-syzygy} for~\(H_{K}^{*}(Y)\) with~\(H_{L}^{*}(Z)\) over~\(\kk\)
  shows that
  \(\HT^{*}(Y\times Z)=H_{K}^{*}(Y)\otimes H_{L}^{*}(Z)\) is again a \(j\)-th syzygy.
  That it cannot be a higher syzygy follows from the characterization of syzygies in terms of regular sequences,
  \cf~\cite[App.~E]{BrunsVetter:1988}:
  Since \(H_{K}^{*}(Y)\) is not a syzygy of order~\(j+1\), there exists a regular sequence~\(f_{1}\),~\ldots,~\(f_{j+1}\) in~\(\RR_{K}\subset\RR\)
  which is not regular for~\(H_{K}^{*}(Y)\). Hence the same sequence cannot be regular for~\(H_{K}^{*}(Y)\otimes H_{L}^{*}(Z)\),
  so \(\HT^{*}(Y\times Z)\) is not a syzygy of order~\(j+1\).
\end{proof}

We have seen in Lemma~\ref{thm:HTX-not-free}
that \(\syzord\HT^{*}(X_{\aa,\bb}(\ell))\) is less than~\(r/2\).
Now we can deduce that big polygon spaces
actually realize all syzygy orders less than~\(r/2\).
In particular, the bound given by Theorem~\ref{thm:syzygy-PD-space} is sharp for any~\(r\).

\begin{corollary}
  \label{thm:HTX-equilateral-0}
  Let \(m\ge0\) and \(r\ge2m+1\) and set
  \begin{equation*}
    \ell=(0,\dots,0,\underbrace{1,\dots,1}_{2m+1})\in\R^{r}.
  \end{equation*}
  Then \(\syzord\HT^{*}(X_{\aa,\bb}(\ell))=m\).  
\end{corollary}

\begin{proof}
  Write~\(k=r-(2m+1)\).
  We have 
  \( 
    X_{\aa,\bb}(\ell) = V^{k}\times X_{\aa,\bb}(1,\dots,1)
  \) by~\eqref{eq:ell-0}. 
  \(H_{K}^{*}(V^{k})\) is free over~\(\RR_{K}\) for the induced action of~\(K=(S^{1})^{k}\),
  for example because the fixed point set is again a product of \(k\)~spheres.
  Now apply Lemma~\ref{rem:even-r}.  
\end{proof}

\section{The general case}
\label{sec:syzygy}

In this section \(\ell\in\R^{r}\) is again a generic length vector,
which we assume to be non-negative and weakly increasing.

\begin{lemma}
  \label{thm:syzord-short-exact}
  Consider the short exact sequence of finitely generated \(\RR\)-modules
  \begin{equation*}
    0 \to M \to M' \to M'' \to 0.
  \end{equation*}
  If \(\syzord M''>\syzord M\), then \(\syzord M'=\syzord M\).
\end{lemma}

\begin{proof}
  Recall that a finitely generated \(\RR\)-module~\(N\) is a \(k\)-th syzygy
  if and only if
  \begin{equation}
    \label{eq:syzygy-depth}
    \depth N_{\pp} \ge \min(k,\depth \RR_{\pp})
  \end{equation}
  for any prime ideal~\(\pp\lhd\RR\) \cite[App.~E]{BrunsVetter:1988}.
  Here `\(\depth N_{\pp}\)'
  refers to the depth over~\(\RR_{\pp}\).
  
  Localizing the given short exact sequence at~\(\pp\),
  we get the short exact sequence
  \begin{equation}
    0 \to M_{\pp} \to M'_{\pp} \to M''_{\pp} \to 0.
  \end{equation}
  Together with~\eqref{eq:syzygy-depth},
  the usual bounds for the depth of modules \cite[Prop.~16.14]{BrunsVetter:1988}
  \begin{align}
    \depth M_{\pp} &\ge \min(\depth M'_{\pp},\depth M''_{\pp}+1), \\
    \depth M'_{\pp} &\ge \min(\depth M_{\pp},\depth M''_{\pp})
  \end{align}
  imply \(\depth M_{\pp}=\depth M'_{\pp}\), which proves the claim.
\end{proof}

\begin{lemma}
  \label{thm:syzygy-HTX-coker}
  \(\syzord\HT^{*}(X_{\aa,\bb}(\ell))=\syzord\coker\iota^{T}_{*}\).
\end{lemma}

\begin{proof}
  Consider the exact sequence
  \begin{equation}
    \label{eq:seq-ker-coker-iota}
    0 \longrightarrow \ker\iota^{T}_{*} \longrightarrow \hHT_{*}(\Vr\setminus X_{\aa,\bb}(\ell))\stackrel{\iota^{T}_{*}}{\longrightarrow} \hHT_{*}(\Vr) \longrightarrow \coker\iota^{T}_{*} \longrightarrow 0,
  \end{equation}
  whose two middle terms are finitely generated free \(\RR\)-modules
  by Lemma~\ref{thm:hHT-free}.
  By splicing \eqref{eq:seq-ker-coker-iota}
  together with the exact sequence~\eqref{eq:def-syzygy} for~\(M=\coker\iota^{T}_{*}\),
  we see that \(\syzord\ker\iota^{T}_{*}\ge\syzord\coker\iota^{T}_{*}+2\).
  The claim thus follows from Lemmas~\ref{thm:extension-HTX} and~\ref{thm:syzord-short-exact}.
\end{proof}

For any~\(J\subset[r]\), define
\begin{equation}
  \sigma_{\ell}(J) = \bigl| \, \{\, j\in J\mid \text{\(J\setminus j\)\, \(\ell\)-short} \,\} \, \bigr|
\end{equation}
and then
\begin{equation}
  \mu(\ell) = \min \{\, \sigma_{\ell}(J) \mid \text{\(J\) \(\ell\)-long and \(\sigma_{\ell}(J)>0\)} \,\}.
\end{equation}
Note that there is always a long~\(J\) such that \(\sigma_{\ell}(J)>0\). For instance,
if \(J\) is a long subset of minimal size, then \(J\ne\emptyset\) and \(\sigma_{\ell}(J)=|J|>0\).

\begin{proposition}
  \label{thm:ell-bound-syzygy}
  \(\syzord\HT^{*}(X_{\aa,\bb}(\ell))\le\mu(\ell)-1 < r/2\).
\end{proposition}

\begin{proof}
  By Lemma~\ref{thm:X-ell-mf},
  the syzygy order of~\(\HT^{*}(X_{\aa,\bb}(\ell))\) depends only on the chamber containing~\(\ell\),
  and the same holds for~\(\mu(\ell)\). So we may assume \(\ell\) to be positive.
  By Lemma~\ref{thm:syzygy-HTX-coker}, to show the first inequality
  it is enough to verify
  that \(M=\coker\iota^{T}_{*}\) is not a syzygy of order~\(k=\mu(\ell)\).
  
  From Proposition~\ref{thm:iota-V-W}
  we know that \(M\) is generated by the \([V_{J}]_{T}\)
  for long~\(J\), subject to the relations
  \begin{equation}
    \label{eq:rel-coker-iota}
    \sum_{\substack{j\notin J\\\text{\(J\cup j\) long}}}
      \negthickspace\negthickspace
      \shufflesign{J}{j} \, \ttt_{j}\cdot [V_{J\cup j}]_{T} = 0
  \end{equation}
  for short~\(J\).
  Because the short sets are the complements of the long ones,
  there is a short~\(J\)
  such that \(J\cup j\) is long only for \(k\)~values~\(j=j_{1}\),~\dots,~\(j_{k}\).
  We claim that the regular sequence~\(\ttt_{j_{1}}\),~\ldots,~\(\ttt_{j_{k}}\in\RR\)
  is not \(M\)-regular:
  The image of~\([V_{j_{k}}]_{T}\) in the quotient~\(M/(\ttt_{j_{1}},\dots,\ttt_{j_{k-1}})M\) is non-zero, but
  \(\ttt_{j_{k}}[V_{j_{k}}]_{T}=0\) there
  because of~\eqref{eq:rel-coker-iota}.
  Hence multiplication by~\(\ttt_{j_{k}}\) is not injective.
  Since there is a regular sequence in~\(\RR\) of length~\(k\) which is not \(M\)-regular,
  \(M\) is not a \(k\)-th syzygy.

  The second inequality, which together with the first
  reproves Lemma~\ref{thm:HTX-not-free}, follows by looking at long subsets of minimal size,
  keeping in mind that half of all subsets are long.
\end{proof}

\begin{corollary} $ $
  \label{thm:syzygy-order-equilateral}
  \begin{enumerate}
  \item Assume that \(r=2m+1\) is odd. Then
    \(\syzord\HT^{*}(X_{\aa,\bb}(\ell))=m\) if and only if
    \(\ell\sim(1,\dots,1)\).
  \item Assume that \(r=2m+2\) is even. Then
    \(\syzord\HT^{*}(X_{\aa,\bb}(\ell))=m\) if and only if
    \(\ell\sim(0,1,\dots,1)\).
  \end{enumerate}
\end{corollary}

Together with Lemma~\ref{thm:X-ell-mf}\,\eqref{thm:X-ell-mf-3} this
proves Theorem~\ref{thm:main}.

\begin{proof}
  We have seen in Section~\ref{sec:equilateral}
  that the length vectors~\((1,\dots,1)\in\R^{2m+1}\) and \((0,1,\dots,1)\in\R^{2m+2}\)
  produce syzygies of order exactly~\(m\).
  
  For the converse, 
  we start with the case of odd~\(r\).
  By Proposition~\ref{thm:ell-bound-syzygy}
  all \(\ell\)-long subsets have size at least~\(m+1\).
  This holds for half of all subsets of~\([r]\).
  At the same time, half of all subsets are long. So we see that
  the \(\ell\)-long subsets are exactly those with at least \(m+1\)~elements.
  Hence \(\ell\sim(1,\dots,1)\).

  We now turn to the even case;
  by Lemma~\ref{thm:X-ell-mf} we may assume \(\ell\) to be strictly increasing.
  The long subsets have again size at least~\(m+1\).
  Among all long subsets~\(J\) of size~\(m+1\), pick one with minimal~\(\ell(J)\).
  (Since half of all subsets are long, it is impossible that there is no long subset of this size.)
  \def\jmin{j_{\rm min}}%
  \def\jmax{j_{\rm max}}%
  Let \(\jmin\) and \(\jmax\) be the minimal and maximal element of~\(J\), respectively.
  Set \(I=J\setminus\{\jmax\}\); it is short. There must be at least~\(m+1\)~elements~\(j\notin I\)
  such that \(I\cup j\) is long for otherwise \(\mu(\ell)\le\sigma_{\ell}(I^{c})<m+1\).
  Since we have chosen a~\(J\) with minimal~\(\ell(J)\), this can only happen if \(\ell_{j}\ge\ell_{\jmax}\),
  \ie, if \(j\ge\jmax\). Hence there are at least~\(m\) elements~\(j\) greater than~\(\jmax\).
  If \(\jmin=1\), then
  \(\ell(J)\le\ell(J^{c})\), so \(J\) would not be long.
  Thus, \(J=\{2,3,\dots,m+2\}\). Since \(J^{c}=\{1,m+3,\dots,r\}\)
  is short, so must be all subsets possibly containing~\(1\) and up to~\(m\) other elements. But these are already half of all subsets.
  We conclude that the short subsets are exactly those which contain at most~\(m\) elements greater than~\(1\).
  This is the same notion of `short' as given by the length vector~\((0,1,\dots,1)\).
\end{proof}

\begin{remark}
  Comparing Corollary~\ref{thm:syzygy-order-equilateral} with Remark~\ref{rem:betti-polygon},
  we see that \(\syzord\HT^{*}(X)\) is largest exactly for those big polygon spaces
  which maximize the Betti sum of the fixed point set. This reminds of the general fact,
  \cf~Section~\ref{sec:proof-main},
  that the equivariant cohomology of a
  \(T\)-manifold~\(Y\) is free (a syzygy of order~\(r\))
  if and only if the Betti sum of~\(Y^{T}\) is as big as possible, namely equal to the Betti sum of~\(Y\).

  However, there seems to be no relation between~\(\syzord\HT^{*}(X)\) and \(\dim H^{*}(X^{T})\)
  for big polygon spaces in general. For example, we have
  \begin{align}
    \syzord\HT^{*}(X_{a,b}(0,0,0,1,1,1)) = 1 &> 0 = \syzord\HT^{*}(X_{a,b}(1,2,2,2,3,3))
  \shortintertext{by Corollary~\ref{thm:HTX-equilateral-0} and Proposition~\ref{thm:ell-bound-syzygy},
  while according to Remark~\ref{rem:betti-polygon}}
    \dim H^{*}(E_{2a}(0,0,0,1,1,1)) = 32 &< 36 = \dim H^{*}(E_{2a}(1,2,2,2,3,3)).    
  \end{align}
\end{remark}

\begin{conjecture}
  \(\syzord\HT^{*}(X_{\aa,\bb}(\ell))=\mu(\ell)-1\).
\end{conjecture}

Using Macaulay2~\cite{M2}
and the lists of chambers computed by Hausmann--Rodri\-guez and Wang
\cite{HausmannRodriguez}, this has been
verified for all chambers in dimensions~\(r\le9\).

\section{Comments}
\label{sec:comments}

\subsection{The mutants}
\label{sec:mutant}

In~\cite[Sec.~4]{FranzPuppe:2008}
Puppe and the author
constructed three examples
of compact orientable \(T\)-manifolds such that \(\syzord\HT^{*}(X)=1\). 
These spaces were called ``mutants of compactified representations''. We sketch a proof that \(Z_{2}\),
the smallest of those examples, is equivariantly homeomorphic to the big polygon space~\(X=X_{1,1}(1,1,1)\).
The mutants for~\(r=5\) and~\(r=9\) are not big polygon spaces, however:
While they have the same dimension and the same Betti sum
as \(X_{1,1}(\ell)\) for~\(\ell\in\R^{r}\), the individual Betti numbers differ.

We start by observing that the quotient~\(X/T\) can be identified with the subspace~\(X_{+}\) of~\(X\)
where all coordinates~\(z_{j}\) are non-negative real numbers. Moreover, the restriction~\(T\times X_{+} \to X\)
of the action
displays \(X\) as an identification space,
\begin{equation}
  \label{eq:X-identification}
  X = \bigl( T \times X_{+} \bigr)\bigm/\mathord\sim\,.
\end{equation}
Here two points~\((g,u,z)\),~\((g',u',z')\in T\times X_{+}\subset T\times\C^{r\aa}\times\C^{r\bb}\)
are identified if~\((u,z)=(u',z')\) and if \(g^{-1}g'\) lies in the coordinate subtorus
\begin{equation}
  \bigl\{\, g\in T = (S^{1})^{3} \bigm| \text{\(g_{j}=1\) if \(z_{j}\ne 0\)} \,\bigr\} \subset T.
\end{equation}

The mutant~\(Z_{2}\) is an identification space by construction,
\begin{equation}
  \label{eq:def-Z}
  Z_{2} = \bigl( T \times Q \bigr)\bigm/\mathord\sim\,.
\end{equation}
Here \(Q\) is a \(4\)-ball, and the non-trivial identifications happen over the \(3\)-sphere~\(\partial Q\)
in the following way: Take a \(2\)-sphere and divide it into three spherical digons. (This is the
boundary of the orbit space of the compactified standard representation of~\(T\) on~\(\C^{3}\)
with its partition into orbit types.)
Lift this partition along the Hopf fibration~\(\partial Q\approx S^{3}\to S^{2}\).
The subtori needed for the identification space~\eqref{eq:def-Z} then
are the isotropy groups occurring in~\(\C^{3}\),
which are again the coordinate subtori of~\(T\).

Hence it is enough to find a homeomorphism between the two orbit spaces \(X_{+}\) and~\(Q\)
that respects the partitions used for the identifications.
Let \(D\subset\C^{3}\) be the unit ball with respect to the maximum norm~%
  \(\|u\|_{\infty} = \max(|u_{1}|,|u_{2}|,|u_{3}|)\).
Since in~\(X_{+}\) all coordinates~\(z_{j}\) are non-negative real numbers,
the projection
\begin{equation}
  X_{+} \to \C^{3},
  \quad
  (u,z)\mapsto u
\end{equation}
is a homeomorphism onto the intersection of~\(D\) with the subspace~\(u_{1}+u_{2}+u_{3}=0\); call it \(P\).
(This is the configuration space of all triangles, possibly degenerate, with sides of length at most~\(1\).)
Note that \(P\) is homeomorphic to a \(4\)-ball, and that the non-trivial identifications in~\eqref{eq:X-identification}
happen exactly over its boundary~\(\partial P\approx S^{3}\).

Now consider the map
\begin{equation}
   p\colon \partial P \to \C\times\R,
  \quad
  u \mapsto \Bigl( \sum_{j=1}^{3}\,(1-|u_{j}|)\,\lambda_{j}, A(u) \Bigr)
\end{equation}
where \(A(u)\) is the oriented area of the polygon with sides~\(u_{1}\),~\(u_{2}\),~\(u_{3}\),
and the \(\lambda_{j}\)'s are the cubic roots of unity.

\begin{lemma}
  The image~\(B\) of~\( p\) is homeomorphic to a \(2\)-sphere,
  and \( p\colon\partial P\to B\) is the Hopf fibration.
\end{lemma}

\begin{proof}
  We start by showing that the image of the map
  \begin{equation}
    \bar p\colon\partial P \to \C,
    \quad
    u \mapsto \sum_{j=1}^{3}\,(1-|u_{j}|)\,\lambda_{j}
  \end{equation}
  is a triangle.
  Since \(u\in\partial P\), at least one the \(u_{j}\)'s has length~\(1\), say \(u_{3}\).
  Then \(\bar p(u)\) lies inside the triangle with vertices~\(\lambda_{1}\),~\(\lambda_{2}\) and the origin.
  In fact, it is the whole triangle because the only
  restriction on the lengths~\(|u_{j}|\in[0,1]\) is the triangle inequality,
  which under the assumption~\(|u_{3}|=1\) translates into the equation
  \begin{equation}
    (1-|u_{1}|) + (1-|u_{2}|) \le 1 + (1-|u_{3}|) = 1.
  \end{equation}
  Thus, the image of~\(\bar p\) is the triangle spanned by the~\(\lambda_{j}\)'s.
  The argument also shows that one can recover the lengths~\(|u_{j}|\) from~\(\bar p(u)\).

  If the lengths are known, there are only two choices for the oriented area~\(A(u)\),
  except when all sides are parallel, in which case there is only one.
  The latter case corresponds to the boundary of the triangle~\(\bar p(\partial P)\).
  Hence the image of~\( p\) consists of two triangles, glued together at their boundaries.
  This gives a \(2\)-sphere.

  It is obvious that \(p\) is invariant under rotations of the complex plane.
  This action is free since at most one~\(u_{j}\) can be zero. Moreover,
  the side lengths and the oriented area determine the triangle up to rotation.
  So \( p\) is a principal \(S^{1}\)-fibration.
  Since domain and codomain are spheres, it must be the Hopf fibration.
\end{proof}

One readily checks that the partition of~\(X_{+}\) by orbit type corresponds to
a partition of~\(B\) into three spherical digons, each containing
both poles and one edge of the triangle forming the equator.
Because this is the same partition as for the mutant,
this proves that \(X\) and~\(Z\) are \(T\)-equivariantly homeomorphic.

\subsection{Connected sums of products of spheres}

Assume \(r=3\) and consider
\(u_{1}\)~and \(u_{2}\) as elements of~\(\R^{2\aa}\).
Introduce new variables~\(\tilde u\in\C^{2\aa}\)
and \(\tilde z_{1}\),~\(\tilde z_{2}\),~\(\tilde z_{3}\in\C^{\bb}\) via
\begin{equation}
  \tilde z_{k} = \frac{z_{k}}{\sqrt{3}}\,,
  \qquad
  \tilde u = \frac{u_{1}+u_{2}}{\sqrt{2}} + i\,\frac{u_{1}-u_{2}}{\sqrt{6}}\,.
\end{equation}
It is elementary, but somewhat tedious to verify that
\(X_{\aa,\bb}(1,1,1)\) can be defined by the equations
\begin{align}
  \label{eq:two-real-quadrics}
  \lambda_{1}\|\tilde z_{1}\|^{2} + \lambda_{2}\|\tilde z_{2}\|^{2} + \lambda_{3}\|\tilde z_{3}\|^{2} + \sum_{l=1}^{2\aa}\tilde u_{l}^{2} &= 0, \\
  \label{eq:unit-sphere}
  \|\tilde z_{1}\|^{2} + \|\tilde z_{2}\|^{2} + \|\tilde z_{3}\|^{2} + \|\tilde u\|^{2} &= 1,
\end{align}
where
\(\lambda_{k}=2\,e^{2\pi k i/3}\)
are the cubic roots of~\(8\).
The intersection of the two real quadrics~\eqref{eq:two-real-quadrics}
has an isolated singularity at the origin, and equation~\eqref{eq:unit-sphere}
exhibits \(X_{\aa,\bb}(1,1,1)\) as its link.

Because the origin is in the interior of the triangle spanned by the~\(\lambda_{k}\)'s,
it follows from a result of
Gómez Gutiérrez and López de Medrano~\cite[Main Thm.]{GomezLopezDeMedrano:2013}
that this manifold is diffeomorphic to a connected sum of products of spheres,
\begin{equation}
  \label{eq:conn-sum}
  X_{\aa,\bb}(1,1,1) \cong \mathop{\#}_{3} S^{2\aa+2\bb-1}\times S^{2\aa+4\bb-2} \, ;
\end{equation}
for~\(\aa=\bb=1\) a homeomorphism of this form was already established in~\cite[Sec.~7]{FranzPuppe:2008}.
This is essentially the only case where this happens,
apart from the trivial case with a single summand, \cf~\eqref{eq:ell-0},
\begin{equation}
  \label{eq:X-zeros-one}
  X_{\aa,\bb}(0,\dots,0,1) = (S^{2\aa+2\bb-1})^{r-1}\times S^{2\bb-1}.
\end{equation}

\begin{proposition}
  If \(X_{\aa,\bb}(\ell)\) has the cohomology algebra of a connected sum of products of spheres,
  then either
  \(\ell\sim(0,\dots,0,1)\) 
  or 
  \(\ell\sim(1,1,1)\).
\end{proposition}

\begin{proof}
  Recall first that the chamber given by~\(\ell=(0,\dots,0,1)\)
  is the only one  for \(r\le2\) (assuming that \(\ell\) is non-negative and weakly-increasing);
  for~\(r=3\) there is exactly one more, given by~\(\ell=(1,1,1)\),
  \cf~\cite[p.~448]{Hausmann:mod2}.
  
  Write \(X=X_{\aa,\bb}(\ell)\).
  By~\eqref{eq:X-zeros-one}
  we can assume \(\ell\not\sim(0,\dots,0,1)\). Then \(r\ge3\), and
  all singleton sets are short.
  Assume
  \begin{equation}
    \label{eq:sum-sphere-products}
    H^{*}(X) \cong H^{*}\bigl( Y_{1} \mathop{\#} \cdots \mathop{\#} Y_{k} \bigr),
  \end{equation}
  where each~\(Y_{i}\) is a product of at least two spheres.
  By Proposition~\ref{thm:betti-Xell},
  the dimension of each sphere must be at least~\(d\), and
  if \(r_{i}\) denotes the number of \(d\)-spheres in~\(Y_{i}\), then
  \(r_{1} + \dots + r_{k} = \dim H^{d}(X)=r\).
  
  As remarked after Proposition~\ref{thm:product-HX-full},
  \(H^{d}(X)\) generates a subalgebra of dimension~\(2^{r-1}\).
  Since each~\(Y_{i}\) is a product of spheres, the subalgebra generated by~\(H^{d}(Y_{i})\)
  has dimension~\(2^{r_{i}}\)
  (and vanishes in degree~\(\dim X\)). Hence \eqref{eq:sum-sphere-products}
  implies
  \begin{equation}
    \label{eq:bound-betti-sum}
    (2^{r_{1}}-1) + \dots + (2^{r_{k}}-1) = 2^{r-1}-1.
  \end{equation}
  Clearly, one solution is \(k=3\) and \(r_{1}=r_{2}=r_{3}=1\).
  In this case we have \(\ell\sim(1,1,1)\) by the remark made at the beginning.
  We claim that there is no other solution.
  
  This claim is obviously true for~\(k=1\). 
  For~\(k=2\) the equation~\eqref{eq:bound-betti-sum}
  is not satisfied if \(r_{1}=1\) or~\(r_{2}=1\) or \(r_{1}=r_{2}=2\).
  In the latter case, the right-hand side dominates, as it does
  for~\(k\ge4\) and \(r_{1}=\dots=r_{k}=1\).
  To finish the proof, it suffices to observe that whenever
  one has the inequality~``\(\le\)'' in~\eqref{eq:bound-betti-sum} and \(r_{i}<r-1\) for some~\(i\),
  then increasing \(r_{i}\) makes the right-hand side dominate strictly.
\end{proof}

\begin{remark}
  One can write any~\(X_{\aa,\bb}(\ell)\) as the link of an intersection of \(r-1\)~homogeneous quadrics:
  The sum of all~\(r\) equations~\eqref{eq:def-X-quadratic} defines a sphere~\(S\) in~\(\C^{r(\aa+\bb)}\),
  and subtracting a multiple of this equation from the other ones makes them homogeneous.
  Eliminating some variables disposes of~\eqref{eq:def-X-linear}.

  Let \(Y\) be the real algebraic variety defined by \(r-1\)~of the homogeneous quadrics thus obtained.
  It is smooth at the points lying on the sphere~\(S\)
  because together with the equation for~\(S\) these quadrics define the manifold~\(X_{\aa,\bb}(\ell)\).
  By homogeneity, this implies that \(Y\) has at most an isolated singularity at the origin,
  with link~\(X_{\aa,\bb}(\ell)\).
\end{remark}

\subsection{Minimal dimension}

We have seen that for any \(m\ge0\)
there are compact orientable \(T\)-manifolds
whose \(T\)-equivariant cohomology
is not free and a syzygy of order exactly~\(m\),
namely the equilateral big polygon spaces~\(X_{\aa,\bb}(1,\dots,1)\).
Here the torus rank is \(r=2m+1\), which is minimal by Theorem~\ref{thm:syzygy-PD-space}.
The dimension of~\(X_{\aa,\bb}(1,\dots,1)\) is at least~\(n=6m+1\); this value is realized for~\(\aa=\bb=1\).
For~\(m=0\) this is clearly the minimum dimension possible as any torus action
on a discrete space is trivial.
More surprisingly, it is also minimal for~\(m=1\). This follows from the bound on the torus rank
together with the following consequence
of the quotient criterion for syzygies in equivariant cohomology \cite[Sec.~7.2]{Franz:orbits3}:

\begin{proposition}
  Let \(X\) be a compact orientable \(T\)-manifold
  such that \(\HT^{*}(X)\) is torsion-free, but not free over~\(\RR \).
  Then \(\dim X\ge 2r+1\).
\end{proposition}

\begin{question}
  For~\(m\ge2\), do examples of maximal syzygies exist in dimension smaller than~\(6m+1\)?
\end{question}


\begin{thebibliography}{99}

\bibitem{AlldayFranzPuppe:orbits1}
C.~Allday, M.~Franz, V.~Puppe,
\newblock Equivariant cohomology, syzygies and orbit structure,
\newblock \textit{Trans.\ Amer.\ Math.\ Soc.}~\textbf{366} (2014), 6567--6589;
\newblock \doi{10.1090/S0002-9947-2014-06165-5}

\bibitem{AlldayFranzPuppe:orbits4}
C.~Allday, M.~Franz, V.~Puppe,
\newblock Equivariant Poincaré--Alexander--Lefschetz duality and the Cohen--Macaulay property,
\newblock \textit{Alg.\ Geom.\ Top.}~\textbf{14} (2014), 1339--1375;
\newblock \doi{10.2140/agt.2014.14.1339}

\bibitem{AlldayPuppe:1993}
C.~Allday, V.~Puppe,
\newblock \textit{Cohomological methods in transformation groups},
\newblock Cambridge Univ.\ Press, Cambridge 1993;
\newblock \doi{10.1017/CBO9780511526275}

\bibitem{BrunsVetter:1988}
W.~Bruns, U.~Vetter,
\newblock \textit{Determinantal rings},
\newblock LNM~\textbf{1327}, Springer, Berlin 1988;
\newblock \doi{10.1007/BFb0080378}
  
\bibitem{FarberFromm:2013}
M.~Farber, V.~Fromm,
\newblock The topology of spaces of polygons,
\newblock \textit{Trans.\ Amer.\ Math.\ Soc.}~\textbf{365} (2013), 3097--3114;
\newblock \doi{10.1090/S0002-9947-2012-05722-9}

\bibitem{FarberSchuetz:2007}
M.~Farber, D.~Schütz,
\newblock Homology of planar polygon spaces,
\newblock \textit{Geom.\ Dedicata}~\textbf{125} (2007), 75--92;
\newblock \doi{10.1007/s10711-007-9139-7}

\bibitem{Franz:orbits3}
M.~Franz,
\newblock A quotient criterion for syzygies in equivariant cohomology,
\newblock \textit{Transformation Groups}~\textbf{22} (2017), 933--965;
\newblock \doi{10.1007/s00031-016-9408-3}

\bibitem{FranzPuppe:2008}
M.~Franz, V.~Puppe,
\newblock Freeness of equivariant cohomology and mutants of compactified representations,
\newblock pp.~87--98 in: M.~Harada \emph{et al.}~(eds.), Toric topology (Osaka, 2006),
  \textit{Contemp.\ Math.}~\textbf{460}, AMS, Providence, RI, 2008;
\newblock \doi{10.1090/conm/460/09012}

\bibitem{Fromm:2011}
V.~Fromm,
\newblock The topology of spaces of polygons,
\newblock doctoral thesis, Durham University 2011;
\newblock \url{http://etheses.dur.ac.uk/3208/}

\bibitem{GomezLopezDeMedrano:2013}
V.~Gómez Gutiérrez, S.~López de Medrano,
\newblock Topology of the intersections of quadrics~II,
\newblock \textit{Bol.\ Soc.\ Mat.\ Mex.}~\textbf{20} (2014), 237--255;
\newblock \doi{10.1007/s40590-014-0038-2}

\bibitem{M2}
D.~R.~Grayson, M.~E.~Stillman,
\newblock Macaulay2,
\newblock \url{http://www.math.uiuc.edu/Macaulay2/}

\bibitem{Hausmann:1991}
J.-C.~Hausmann,
\newblock Sur la topologie des bras articulés,
\newblock pp.~146--159 in: S.~Jackowski \emph{et al.}~(eds.), Algebraic topology (Poznań, 1989),
\newblock LNM~\textbf{1474}, Springer, Berlin 1991;
\newblock \doi{10.1007/BFb0084743}

\bibitem{Hausmann:mod2}
J.-C.~Hausmann,
\newblock \textit{Mod two homology and cohomology},
\newblock Springer, Cham 2014;
\newblock \doi{10.1007/978-3-319-09354-3}

\bibitem{HausmannRodriguez}
J.-C.~Hausmann, E.~Rodriguez,
\newblock Corrections and additional material to ``The space of clouds in Euclidean space'',
\newblock \url{http://www.unige.ch/math/folks/hausmann/polygones/}

\bibitem{Puppe:2018}
V.~Puppe,
\newblock Equivariant cohomology of $(\Z_{2})^{r}$‑manifolds and syzygies,
\newblock \textit{Fund.\ Math.}~\textbf{243} (2018), 55--74;
\newblock \doi{10.4064/fm405-12-2017}

\end{thebibliography}
\end{document}